\pdfoutput=1
\documentclass[reqno]{amsart}

\usepackage[utf8]{inputenc}
\usepackage[english]{babel}

\usepackage{soul}
\usepackage{braket}
\usepackage{amsmath}
\usepackage{amstext}
\usepackage{amssymb}
\usepackage{amsthm}
\usepackage{mathtools}
\usepackage{stmaryrd}
\usepackage{mathrsfs}
\usepackage{extarrows}
\usepackage{afterpage}
\usepackage{tabu}
\usepackage{array}
\usepackage[foot]{amsaddr}
\usepackage{microtype}
\usepackage{graphicx}
\usepackage{amsfonts}
\usepackage{units}
\usepackage[shortlabels]{enumitem}
\usepackage{url}
\usepackage{bm}
\usepackage{tikz}
\usetikzlibrary{matrix,arrows,calc,fit,cd,positioning,intersections,arrows.meta,decorations.markings,decorations.pathmorphing,shapes,calligraphy}
\usepackage{tikz-3dplot}
\usepackage{subcaption}

\usepackage{hyperref}

\usepackage{xcolor}
\hypersetup{
    colorlinks,
    linkcolor={red!50!black},
    citecolor={blue!50!black},
    urlcolor={blue!80!black}
}

\usepackage{cleveref}
\usepackage[margin=1cm]{caption}
\usepackage{colortbl}
\usepackage{framed}

\usepackage{cleveref}

\allowdisplaybreaks[4]

\theoremstyle{definition}
\newtheorem{theorem}{Theorem}[section]
\newtheorem{definition}[theorem]{Definition}
\newtheorem{remark}[theorem]{Remark}

\newtheorem{lemma}[theorem]{Lemma}
\newtheorem{proposition}[theorem]{Proposition}
\newtheorem{corollary}[theorem]{Corollary}
\newtheorem{conjecture}[theorem]{Conjecture}

\newtheorem*{mainthm*}{Main Theorem}
\newtheorem{mainthm}{Main Theorem}

\newcommand{\N}{\mathbb{N}}
\newcommand{\Z}{\mathbb{Z}}
\newcommand{\R}{\mathbb{R}}
\newcommand{\C}{\mathbb{C}}
\renewcommand{\H}{\mathbb{H}}

\newcommand{\M}{\mathcal{M}}

\newcommand{\F}{\mathcal{F}}

\newcommand{\CC}{\mathcal{C}}

\DeclareMathOperator{\id}{id}
\DeclareMathOperator{\rk}{rk}
\DeclareMathOperator{\im}{im}

\DeclareMathOperator{\PGL}{PGL}

\newcommand{\<}{\langle}
\renewcommand{\>}{\rangle}

\DeclareMathOperator{\Mov}{\textsc{Mov}}
\DeclareMathOperator{\Fix}{\textsc{Fix}}
\DeclareMathOperator{\Min}{\textsc{Min}}

\DeclareMathOperator{\Isom}{\textsc{Isom}}

\newcommand{\Ksparse}[1][W]{K_{#1}^{\textrm{sparse}}}	%
\newcommand{\Kshort}[1][W]{K_{#1}^{\textrm{short}}}		%
\newcommand{\Klong}[1][W]{K_{#1}^{\textrm{long}}}		%

\newenvironment{mymatrix}[1]{
    \left(
    
    \begin{array}{#1}
}
{
    \end{array}
    \right)
}

\newcommand{\hgline}[3]{
	\pgfmathsetmacro{\thetaone}{mod(#1,360)}
	\pgfmathsetmacro{\thetatwo}{mod(#2,360)}
	\pgfmathsetmacro{\theta}{(\thetaone+\thetatwo)/2}
	\pgfmathsetmacro{\phi}{abs(\thetaone-\thetatwo)/2}
	\pgfmathsetmacro{\close}{less(abs(\phi-90),0.0001)}
	\ifdim \close pt = 1pt
	\draw (\thetaone:1) -- (\thetatwo:1);
	\node at (\thetatwo:0.2) {#3};
	\else
	\pgfmathsetmacro{\R}{tan(\phi)}
	\ifdim \R pt < 0pt
	\pgfmathsetmacro{\distance}{sqrt(1+\R*\R)}
	\draw (\theta:-\distance) circle (\R);
	\node at (\theta:-\distance-\R) {#3};
	\else \ifdim \R pt > 0pt
	\pgfmathsetmacro{\distance}{sqrt(1+\R^2)}
	\draw (\theta:\distance) circle (\R);
	\node at (\theta:\distance-\R) {#3};
	\fi
	\fi
	\fi
}

\title[Dual structures on Coxeter and Artin groups of rank three]{
Dual structures on Coxeter and Artin groups \\ of rank three
}

\author[Emanuele Delucchi]{Emanuele Delucchi}
\address{\textnormal{Emanuele Delucchi: SUPSI-IDSIA, University of Applied Arts and Sciences of Southern Switzerland, Lugano, Switzerland.
}}

\author[Giovanni Paolini]{Giovanni Paolini}
\address{\textnormal{Giovanni Paolini (corresponding author): Department of Mathematics, University of Bologna, Italy.}}

\author[Mario Salvetti]{Mario Salvetti}
\address{\textnormal{Mario Salvetti: Department of Mathematics, University of Pisa, Italy.}}

\email{emanuele.delucchi@supsi.ch\textnormal{,} g.paolini@unibo.it\textnormal{,} salvetti@dm.unipi.it}

\begin{document}

\begin{abstract}
    We extend the theory of dual Coxeter and Artin groups to all rank-three Coxeter systems, beyond the previously studied spherical and affine cases.
    Using geometric, combinatorial, and topological techniques, we show that rank-three noncrossing partition posets are EL-shellable lattices and give rise to Garside groups isomorphic to the associated standard Artin groups.
    Within this framework, we prove the $K(\pi, 1)$ conjecture, the triviality of the center, and the solubility of the word problem for rank-three Artin groups.
    Some of our constructions apply to general Artin groups; we hope they will help develop complete solutions to the $K(\pi, 1)$ conjecture and other open problems in the area.
\end{abstract}

\maketitle

\section{Introduction}
\label{sec:introduction}

The famous $K(\pi, 1)$ conjecture, dating back to the 1960s and due to Arnol'd, Pham, and Thom, states that the orbit configuration space of Artin groups is an Eilenberg-Maclane space (or $K(\pi, 1)$ space). This conjecture was proved fifty years ago by Deligne for spherical Artin groups \cite{deligne1972immeubles}. Recently, the second and third authors proved this conjecture for the next important class of Artin groups, namely those of affine type \cite{paolini2021proof}.
Between these two results, Dehornoy, Paris, and others developed the theory of Garside structures \cite{bestvina1999,dehornoy1999gaussian,bessis2003dual,charney2004bestvina,dehornoy2015foundations}, generalizing properties of the standard presentation of spherical Artin groups.
Applying Garside theory to \emph{dual} presentations (arising from noncrossing partition posets) was a key ingredient to solving the most important open problems on affine Artin groups: not only the $K(\pi, 1)$ conjecture \cite{paolini2021proof}, but also the word problem and the triviality of the center \cite{mccammond2017artin}.
An outline of this ``dual approach'', inspired by \cite{paolini2021proof} and articulated in \cite{paolini2021dual}, is given in \Cref{sec:dual-approach}.

We begin the extension of the dual approach beyond the affine case by tackling Artin groups associated with hyperbolic Coxeter systems of rank $3$ (see \Cref{fig:hyperbolic-arrangements}).
While these Artin groups were already partially understood with different techniques \cite{hendriks1985hyperplane, charney1995k, chermak1998locally}, we are able to solve all important problems about them (including the ones previously solved) independently of prior work.
We also answer all questions posed in \cite{paolini2021dual} on the dual structure of rank-three Artin groups.
Our main results are summarized in \Cref{sec:contributions} below.

\begin{figure}
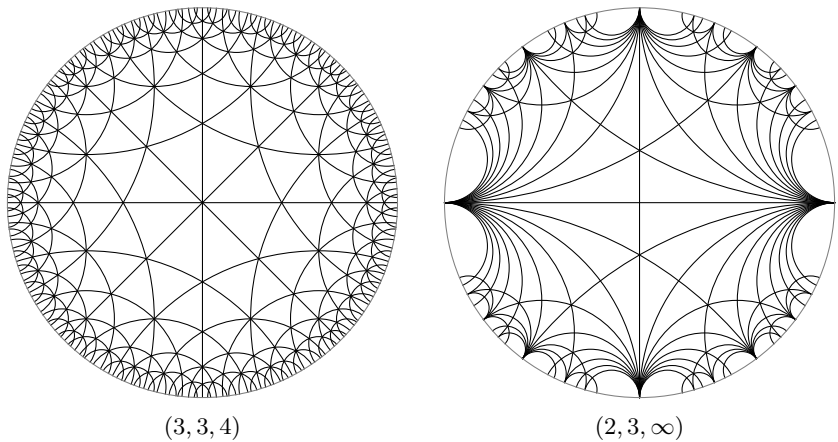

    \newcommand{\scale}{2.6}
    \begin{center}
        \begin{subfigure}[t]{.45\linewidth}
            \centering
            \begin{tikzpicture}[scale=\scale]
	            \clip (0,0) circle (1);
	            \input{lines334}
	            \draw[black!50, thick] (0,0) circle (1);
            \end{tikzpicture}
            \caption*{$(3, 3, 4)$}
        \end{subfigure}
        \begin{subfigure}[t]{.45\linewidth}
            \centering
            \begin{tikzpicture}[scale=\scale]
	            \clip (0,0) circle (1);
	            \input{lines23infty}
	            \draw[black!50, thick] (0,0) circle (1);
            \end{tikzpicture}
            \caption*{$(2, 3, \infty)$}
        \end{subfigure}
    \end{center}
    
    \caption{Examples of arrangements of rank-three hyperbolic Coxeter groups in the Poincaré model.
    Each picture is captioned with the labels $(m_1, m_2, m_3)$ of the corresponding Coxeter diagram.
    The hyperbolic plane is tiled by triangles with angles $\frac{\pi}{m_1}$, $\frac{\pi}{m_2}$, and $\frac{\pi}{m_3}$.
 	}
    \label{fig:hyperbolic-arrangements}
\end{figure}

\begin{figure}
    \makeatletter
    \define@key{z sphericalkeys}{radius}{\def\myradius{#1}}
    \define@key{z sphericalkeys}{theta}{\def\mytheta{#1}}
    \define@key{z sphericalkeys}{phi}{\def\myphi{#1}}
    \tikzdeclarecoordinatesystem{z spherical}{%
        \setkeys{z sphericalkeys}{#1}%
        \pgfpointxyz{\myradius*sin(\mytheta)*cos(\myphi)}{\myradius*sin(\mytheta)*sin(\myphi)}{\myradius*cos(\mytheta)}}
    \makeatother
    
    \begin{center}
    \begin{subfigure}[t]{.3\linewidth}
        \centering
        \begin{tikzpicture}[scale=1.9]
            \tdplotsetmaincoords{10}{20}
            \begin{scope}[tdplot_main_coords]
                \tdplotsetrotatedcoords{0}{90}{0}
                \begin{scope}[tdplot_rotated_coords]
                    \newcommand{\beginphi}{96}
                    \draw plot[variable=\x,domain=\beginphi:\beginphi+180,smooth,samples=60]  (z spherical cs:radius=1,theta=90,phi=\x);
                \end{scope}
                
                \tdplotsetrotatedcoords{90}{90}{0}
                \begin{scope}[tdplot_rotated_coords]
                    \newcommand{\beginphi}{95}
                    \draw plot[variable=\x,domain=\beginphi:\beginphi+180,smooth,samples=60]  (z spherical cs:radius=1,theta=90,phi=\x);
                \end{scope}
                
                \tdplotsetrotatedcoords{45}{54.7356103}{0}
                \begin{scope}[tdplot_rotated_coords]
                    \newcommand{\beginphi}{100}
                    \draw plot[variable=\x,domain=\beginphi:\beginphi+180,smooth,samples=60]  (z spherical cs:radius=1,theta=90,phi=\x);
                \end{scope}
                
                \tdplotsetrotatedcoords{-45}{54.7356103}{0}
                \begin{scope}[tdplot_rotated_coords]
                    \newcommand{\beginphi}{100}
                    \draw plot[variable=\x,domain=\beginphi:\beginphi+180,smooth,samples=60]  (z spherical cs:radius=1,theta=90,phi=\x);
                \end{scope}
                
                \tdplotsetrotatedcoords{45}{125.2643897}{0}
                \begin{scope}[tdplot_rotated_coords]
                    \newcommand{\beginphi}{100}
                    \draw plot[variable=\x,domain=\beginphi:\beginphi+180,smooth,samples=60]  (z spherical cs:radius=1,theta=90,phi=\x);
                \end{scope}
        
                \tdplotsetrotatedcoords{-45}{125.2643897}{0}
                \begin{scope}[tdplot_rotated_coords]
                    \newcommand{\beginphi}{95}
                    \draw plot[variable=\x,domain=\beginphi:\beginphi+180,smooth,samples=60]  (z spherical cs:radius=1,theta=90,phi=\x);
                \end{scope}
            \end{scope}
            \draw[black!50, thick] (0,0,0) circle (1);
        \end{tikzpicture}
        \caption*{$A_3$: $(2,3,3)$}
    \end{subfigure}
    \quad
    \begin{subfigure}[t]{.3\linewidth}
        \centering
        \begin{tikzpicture}[scale=1.9]
            \tdplotsetmaincoords{30}{30}
            \begin{scope}[tdplot_main_coords]
                \tdplotsetrotatedcoords{0}{90}{0}
                \begin{scope}[tdplot_rotated_coords]
                    \newcommand{\beginphi}{115}
                    \draw plot[variable=\x,domain=\beginphi:\beginphi+180,smooth,samples=60]  (z spherical cs:radius=1,theta=90,phi=\x);
                \end{scope}
    
                \tdplotsetrotatedcoords{90}{90}{0}
                \begin{scope}[tdplot_rotated_coords]
                    \newcommand{\beginphi}{100}
                    \draw plot[variable=\x,domain=\beginphi:\beginphi+180,smooth,samples=60]  (z spherical cs:radius=1,theta=90,phi=\x);
                \end{scope}
    
                \tdplotsetrotatedcoords{0}{0}{0}
                \begin{scope}[tdplot_rotated_coords]
                    \newcommand{\beginphi}{200}
                    \draw plot[variable=\x,domain=\beginphi:\beginphi+180,smooth,samples=60]  (z spherical cs:radius=1,theta=90,phi=\x);
                \end{scope}
    
                \tdplotsetrotatedcoords{45}{90}{0}
                \begin{scope}[tdplot_rotated_coords]
                    \newcommand{\beginphi}{117}
                    \draw plot[variable=\x,domain=\beginphi:\beginphi+180,smooth,samples=60]  (z spherical cs:radius=1,theta=90,phi=\x);
                \end{scope}
    
                \tdplotsetrotatedcoords{-45}{90}{0}
                \begin{scope}[tdplot_rotated_coords]
                    \newcommand{\beginphi}{93}
                    \draw plot[variable=\x,domain=\beginphi:\beginphi+180,smooth,samples=60]  (z spherical cs:radius=1,theta=90,phi=\x);
                \end{scope}
    
                \tdplotsetrotatedcoords{90}{45}{0}
                \begin{scope}[tdplot_rotated_coords]
                    \newcommand{\beginphi}{100}
                    \draw plot[variable=\x,domain=\beginphi:\beginphi+180,smooth,samples=60]  (z spherical cs:radius=1,theta=90,phi=\x);
                \end{scope}
    
                \tdplotsetrotatedcoords{-90}{45}{0}
                \begin{scope}[tdplot_rotated_coords]
                    \newcommand{\beginphi}{60}
                    \draw plot[variable=\x,domain=\beginphi:\beginphi+180,smooth,samples=60]  (z spherical cs:radius=1,theta=90,phi=\x);
                \end{scope}
    
                \tdplotsetrotatedcoords{0}{45}{0}
                \begin{scope}[tdplot_rotated_coords]
                    \newcommand{\beginphi}{130}
                    \draw plot[variable=\x,domain=\beginphi:\beginphi+180,smooth,samples=60]  (z spherical cs:radius=1,theta=90,phi=\x);
                \end{scope}
    
                \tdplotsetrotatedcoords{0}{135}{0}
                \begin{scope}[tdplot_rotated_coords]
                    \newcommand{\beginphi}{120}
                    \draw plot[variable=\x,domain=\beginphi:\beginphi+180,smooth,samples=60]  (z spherical cs:radius=1,theta=90,phi=\x);
                \end{scope}
            \end{scope}
            \draw[black!50, thick] (0,0,0) circle (1);
        \end{tikzpicture}
        \caption*{$B_3$: $(2,3,4)$}
    \end{subfigure}
    \quad
    \begin{subfigure}[t]{.3\linewidth}
        \centering
        \begin{tikzpicture}[scale=1.9]
    
            \tdplotsetmaincoords{32}{152}
            \begin{scope}[tdplot_main_coords]
                \tdplotsetrotatedcoords{0}{90}{0}
                \begin{scope}[tdplot_rotated_coords]
                    \newcommand{\beginphi}{60}
                    \draw plot[variable=\x,domain=\beginphi:\beginphi+180,smooth,samples=60]  (z spherical cs:radius=1,theta=90,phi=\x);
                \end{scope}
    
                \tdplotsetrotatedcoords{90}{90}{0}
                \begin{scope}[tdplot_rotated_coords]
                    \newcommand{\beginphi}{110}
                    \draw plot[variable=\x,domain=\beginphi:\beginphi+180,smooth,samples=60]  (z spherical cs:radius=1,theta=90,phi=\x);
                \end{scope}
    
                \tdplotsetrotatedcoords{0}{0}{0}
                \begin{scope}[tdplot_rotated_coords]
                    \newcommand{\beginphi}{-30}
                    \draw plot[variable=\x,domain=\beginphi:\beginphi+180,smooth,samples=60]  (z spherical cs:radius=1,theta=90,phi=\x);
                \end{scope}
                
                \tdplotsetrotatedcoords{31.717474411460998}{72}{0}
                \begin{scope}[tdplot_rotated_coords]
                    \newcommand{\beginphi}{70}
                    \draw plot[variable=\x,domain=\beginphi:\beginphi+180,smooth,samples=60]  (z spherical cs:radius=1,theta=90,phi=\x);
                \end{scope}
    
                \tdplotsetrotatedcoords{148.282525588539}{72}{0}
                \begin{scope}[tdplot_rotated_coords]
                    \newcommand{\beginphi}{120}
                    \draw plot[variable=\x,domain=\beginphi:\beginphi+180,smooth,samples=60]  (z spherical cs:radius=1,theta=90,phi=\x);
                \end{scope}
    
                \tdplotsetrotatedcoords{-31.717474411460998}{72}{0}
                \begin{scope}[tdplot_rotated_coords]
                    \newcommand{\beginphi}{60}
                    \draw plot[variable=\x,domain=\beginphi:\beginphi+180,smooth,samples=60]  (z spherical cs:radius=1,theta=90,phi=\x);
                \end{scope}
    
                \tdplotsetrotatedcoords{-148.282525588539}{72}{0}
                \begin{scope}[tdplot_rotated_coords]
                    \newcommand{\beginphi}{110}
                    \draw plot[variable=\x,domain=\beginphi:\beginphi+180,smooth,samples=60]  (z spherical cs:radius=1,theta=90,phi=\x);
                \end{scope}
    
                \tdplotsetrotatedcoords{69.09484255211069}{60}{0}
                \begin{scope}[tdplot_rotated_coords]
                    \newcommand{\beginphi}{90}
                    \draw plot[variable=\x,domain=\beginphi:\beginphi+180,smooth,samples=60]  (z spherical cs:radius=1,theta=90,phi=\x);
                \end{scope}
    
                \tdplotsetrotatedcoords{110.90515744788931}{60}{0}
                \begin{scope}[tdplot_rotated_coords]
                    \newcommand{\beginphi}{130}
                    \draw plot[variable=\x,domain=\beginphi:\beginphi+180,smooth,samples=60]  (z spherical cs:radius=1,theta=90,phi=\x);
                \end{scope}
    
                \tdplotsetrotatedcoords{-69.09484255211069}{60}{0}
                \begin{scope}[tdplot_rotated_coords]
                    \newcommand{\beginphi}{65}
                    \draw plot[variable=\x,domain=\beginphi:\beginphi+180,smooth,samples=60]  (z spherical cs:radius=1,theta=90,phi=\x);
                \end{scope}
    
                \tdplotsetrotatedcoords{-110.90515744788931}{60}{0}
                \begin{scope}[tdplot_rotated_coords]
                    \newcommand{\beginphi}{80}
                    \draw plot[variable=\x,domain=\beginphi:\beginphi+180,smooth,samples=60]  (z spherical cs:radius=1,theta=90,phi=\x);
                \end{scope}
    
                \tdplotsetrotatedcoords{31.71747441146101}{36}{0}
                \begin{scope}[tdplot_rotated_coords]
                    \newcommand{\beginphi}{35}
                    \draw plot[variable=\x,domain=\beginphi:\beginphi+180,smooth,samples=60]  (z spherical cs:radius=1,theta=90,phi=\x);
                \end{scope}
    
                \tdplotsetrotatedcoords{148.282525588539}{36}{0}
                \begin{scope}[tdplot_rotated_coords]
                    \newcommand{\beginphi}{140}
                    \draw plot[variable=\x,domain=\beginphi:\beginphi+180,smooth,samples=60]  (z spherical cs:radius=1,theta=90,phi=\x);
                \end{scope}
    
                \tdplotsetrotatedcoords{-31.71747441146101}{36}{0}
                \begin{scope}[tdplot_rotated_coords]
                    \newcommand{\beginphi}{40}
                    \draw plot[variable=\x,domain=\beginphi:\beginphi+180,smooth,samples=60]  (z spherical cs:radius=1,theta=90,phi=\x);
                \end{scope}
    
                \tdplotsetrotatedcoords{-148.282525588539}{36}{0}
                \begin{scope}[tdplot_rotated_coords]
                    \newcommand{\beginphi}{110}
                    \draw plot[variable=\x,domain=\beginphi:\beginphi+180,smooth,samples=60]  (z spherical cs:radius=1,theta=90,phi=\x);
                \end{scope}
    
            \end{scope}
            \draw[black!50, thick] (0,0,0) circle (1);
        \end{tikzpicture}
        \caption*{$H_3$: $(2, 3, 5)$}
    \end{subfigure}
    \end{center}
    
    \bigskip
    
    \begin{center}
    \newcommand{\scale}{0.8}
    \begin{subfigure}[t]{.3\linewidth}
        \centering
        \newcommand*\rows{10}
        \begin{tikzpicture}[scale=\scale,
          extended line/.style={shorten >=-#1,shorten <=-#1},
          extended line/.default=35cm]
            \begin{scope}
                \clip (-2.2, -2.2) rectangle (2.2, 2.2);
                \foreach \row in {-\rows, ...,\rows} {
                    \draw [extended line] ($\row*(0.5, {0.5*sqrt(3)})$) -- ($(\rows,0)+\row*(-0.5, {0.5*sqrt(3)})$);
                    \draw [extended line] ($\row*(1, 0)$) -- ($(\rows/2,{\rows/2*sqrt(3)})+\row*(0.5,{-0.5*sqrt(3)})$);
                    \draw [extended line] ($\row*(1, 0)$) -- ($(\rows/2,{-\rows/2*sqrt(3)})+\row*(0.5,{0.5*sqrt(3)})$);
                }
            \end{scope}
        \end{tikzpicture}
        \caption*{$\tilde A_2$: $(3,3,3)$}
    \end{subfigure}
    \quad
    \begin{subfigure}[t]{.3\linewidth}
        \centering
        \newcommand*\rows{10}
        \begin{tikzpicture}[scale=\scale,
          extended line/.style={shorten >=-#1,shorten <=-#1},
          extended line/.default=50cm]
            \begin{scope}
                \clip (-2.2, -2.2) rectangle (2.2, 2.2);
                \foreach \row in {-\rows, ...,\rows} {
                    \draw [extended line] ($\row*(1, 1)$) -- ($(\rows,0)+\row*(1, 1)$);
                    \draw [extended line] ($\row*(1, 1)$) -- ($(0,\rows)+\row*(1, 1)$);
                    \draw [extended line] ($\row*(1, 0)$) -- ($(1, 1)+\row*(1, 0)$);
                    \draw [extended line] ($\row*(1, 0)$) -- ($(1, -1)+\row*(1, 0)$);
                }
            \end{scope}
        \end{tikzpicture}
        \caption*{$\tilde C_2$: $(2,4,4)$}
    \end{subfigure}
    \quad
    \begin{subfigure}[t]{.3\linewidth}
        \centering
        \newcommand*\rows{10}
        \begin{tikzpicture}[scale=\scale,
          extended line/.style={shorten >=-#1,shorten <=-#1},
          extended line/.default=35cm]
            \begin{scope}
                \clip (-2.2, -2.2) rectangle (2.2, 2.2);
                \foreach \row in {-\rows, ...,\rows} {
                    \draw [extended line] ($\row*(0, 1)$) -- ($(1, 0)+\row*(0, 1)$);
                    \draw [extended line] ($\row*({sqrt(3)},0)$) -- ($(0, 1)+\row*({sqrt(3)},0)$);
                    \draw [extended line] ($\row*(0, 2)$) -- ($({sqrt(3)}, 3)+\row*(0, 2)$);
                    \draw [extended line] ($\row*(0, 2)$) -- ($({sqrt(3)}, 1)+\row*(0, 2)$);
                    \draw [extended line] ($\row*(0, 2)$) -- ($({sqrt(3)}, -3)+\row*(0, 2)$);
                    \draw [extended line] ($\row*(0, 2)$) -- ($({sqrt(3)}, -1)+\row*(0, 2)$);
                }
            \end{scope}
        \end{tikzpicture}
        \caption*{$\tilde G_2$: $(2,3,6)$}
    \end{subfigure}
    \end{center}
    
    \caption{Arrangements of irreducible spherical and affine Coxeter groups of rank three.
    Each picture is captioned with the standard name (following the classification of spherical and affine Coxeter groups) and the labels $(m_1, m_2, m_3)$ of the corresponding Coxeter diagram.
    In each case, the sphere or plane is tiled by triangles with angles $\frac{\pi}{m_1}, \frac{\pi}{m_2}, \frac{\pi}{m_3}$.
    The Coxeter groups of the spherical cases (top three) are the symmetry groups of the five platonic solids (tetrahedron, cube/octahedron, and dodecahedron/icosahedron, respectively).}
    \label{fig:spherical-affine-arrangements}
\end{figure}
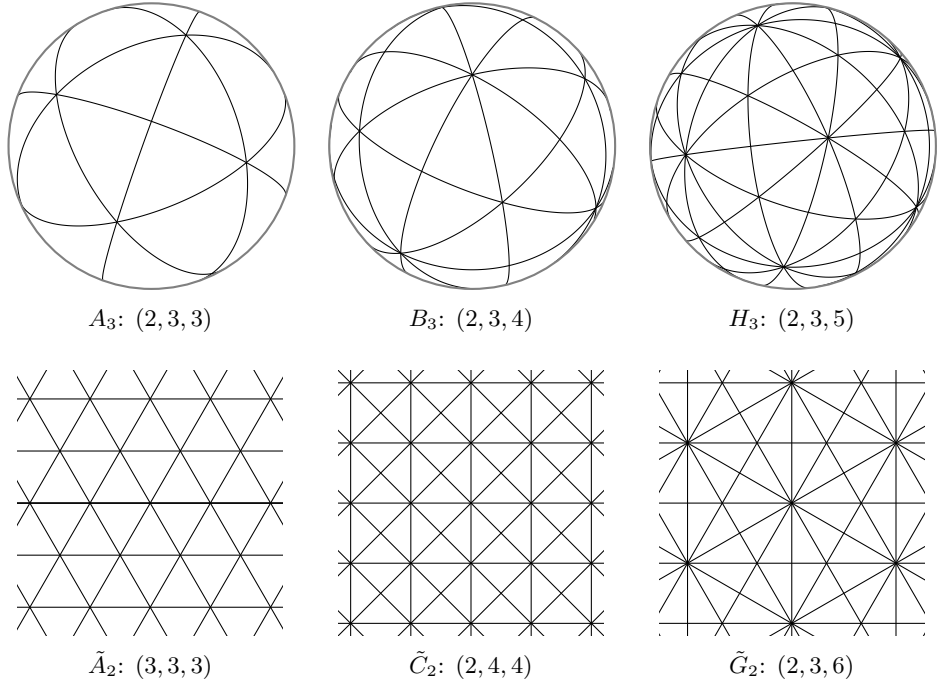

\subsection{The dual approach}
\label{sec:dual-approach}
Let $W$ be a finitely generated Coxeter group and $G_W$ the associated Artin group.
In the affine case, the proof of the $K(\pi, 1)$ conjecture given in \cite{paolini2021proof} roughly goes as follows.
First, the so-called \emph{interval complex} $K_W$ (see below) is shown to be a $K(\pi, 1)$ space.
The next step is to identify a finite subcomplex $X'_W\subseteq K_W$ and to prove that it is homotopy equivalent to the orbit configuration space of $G_W$.
Finally, using combinatorial methods (discrete Morse theory and lexicographic shellability), the complex $K_W$ is shown to deformation retract onto $X'_W$.

The noncrossing partition poset $[1,w]$ is the interval between $1$ and $w$ in the (right) Cayley graph of $W$, using the set $R$ of all reflections as generators.
The interval complex $K_W$ is a quotient of the order complex of $[1,w]$; it is a $\Delta$-complex whose $d$-simplices correspond to sequences $\sigma=[x_1|\dotsb|x_d]$ which are part of a reduced factorization of $w$.
One also introduces a \emph{dual} Artin group $W_w$ associated with $W$ and the chosen Coxeter element $w$, defined as the fundamental group of $K_W$.
Combinatorial properties of $[1,w]$ reflect topological properties of $K_W$: most significantly, if $[1,w]$ is a lattice then $K_W$ is a $K(\pi,1)$.

In addition to the $K(\pi, 1)$ conjecture, several interesting questions can be asked in general and have been answered in the affine case.
For example: is the dual Artin group $W_w$ always naturally isomorphic to the standard one? Is it a Garside group? Is the word problem solvable?
And, from the combinatorial point of view: is the noncrossing partition poset $[1,w]$ a lattice? is it EL-shellable?
We refer to \cite{paolini2021dual} for a more detailed discussion on the dual approach and the several questions related to it.

\subsection{Contributions}
\label{sec:contributions}
By studying the geometry and combinatorics of the noncrossing partition poset $[1,w]$, we answer all previous questions (and more) when $W$ is a Coxeter group of rank $3$, i.e., generated by $3$ reflections.
These groups are also known as \emph{triangle groups}.
They are all of hyperbolic type in the sense of \cite{humphreys1992reflection} (\Cref{fig:hyperbolic-arrangements}) except for a finite list of spherical or affine groups that are already well understood (\Cref{fig:spherical-affine-arrangements}).

Going from affine to hyperbolic groups introduces new challenges.
For example, a geometric characterization of the elements of $[1,w]$ seems difficult to achieve (in the affine case, such a characterization was obtained by McCammond in \cite{mccammond2015dual}).
The rank-three case is also peculiar because $[1,w]$ turns out to be a lattice, thus giving rise to a Garside structure on the corresponding dual Artin group. The lattice property does not hold in general, even in the affine case \cite[Theorem A]{mccammond2015dual}.
Aside from the lattice property, we expect many of the results and techniques we develop to extend to hyperbolic Coxeter groups of arbitrary rank.
The following are our main results.

\begin{mainthm}[Dual structure]
    Let $(W, S)$ be a Coxeter system of rank $3$ and $G_W$ the associated Artin group.
    Let $w \in W$ be any Coxeter element, and consider the associated noncrossing partition poset $[1,w]$.
    \begin{enumerate}[(i)]
        \item $[1,w]$ is a lattice.
        \item $[1,w]$ is EL-shellable.
        \item Every element $u \in [1,w]$ is a Coxeter element for the subgroup generated by the reflections $\leq u$.
        \item The dual Artin group associated with $[1,w]$ is naturally isomorphic to $G_W$.
    \end{enumerate}
\end{mainthm}

\begin{mainthm}[Artin groups]
    Let $G_W$ be an Artin group of rank $3$.
    \begin{enumerate}[(i)]
        \item $G_W$ is a Garside group.
        \item The $K(\pi, 1)$ conjecture holds for $G_W$.
        \item The word problem for $G_W$ is solvable.
        \item The center of $G_W$ is trivial unless $W$ is finite.
    \end{enumerate}
    \label{thm:main2}
\end{mainthm}

Some of the claims in \Cref{thm:main2} have been proved elsewhere by completely different methods, but our setup allows us to obtain particularly succinct proofs for all of them (see \Cref{sec:consequences}).
Specifically, the $K(\pi, 1)$ conjecture \cite{hendriks1985hyperplane, charney1995k} and the word problem \cite{chermak1998locally} were already known for rank-three Artin groups.
During the preparation of this paper, a preprint appeared showing that the $K(\pi, 1)$ conjecture implies the triviality of the center \cite{jankiewicz2022conjecture} for general Artin groups.
All other results are completely novel.

In order to prove the theorems above, we further develop the dual approach by providing new constructions that hold for general Artin groups. In particular, we introduce new subcomplexes of the interval complex $K_W$ and propose a general strategy to deformation retract $K_W$ onto $X_W'$ (among other things, this would imply the isomorphism between standard and dual Artin groups).

\subsection{Structure}
In \Cref{sec:preliminaries}, we give all definitions and constructions needed later: classical models of the hyperbolic plane; some standard tools of combinatorial topology; interval groups and their relation with Garside structures; Coxeter groups, (standard and dual) Artin groups, their orbit configuration space and the statement of the $K(\pi,1)$ conjecture; and the definition of the subcomplex $X'_W$, homotopy equivalent to the orbit configuration space.

In \Cref{sec:coxeter-elements}, we introduce Coxeter elements of rank-three hyperbolic groups and their Coxeter axes. The main result there is that, if $w$ is a Coxeter element,  every element $u$ in the interval $[1,w]$ is a Coxeter element for the subgroup of $W$ generated by all reflections below $u$ (\Cref{thm:coxeter-elements-below-w}).

In \Cref{sec:poset}, we prove the lattice property for $[1,w]$, which implies that the corresponding dual Artin group is a Garside group.
Then, we introduce the \emph{axial ordering} on the set of reflections in $[1,w]$. The main result (\Cref{thm:shellability}) is that this ordering induces an EL-labeling of $[1,w]$.

In \Cref{sec:interval-complex}, we introduce a sequence of subcomplexes of the interval complex $K_W$, for an arbitrary Coxeter group $W$.
These include the already mentioned subcomplex $X_W'$ (homotopy equivalent to the orbit configuration space) as well as new intermediate subcomplexes between $X_W'$ and $K_W$.
We hypothesize that a deformation retraction $K_W \searrow X_W'$ can be constructed more easily
by collapsing each intermediate subcomplex to the next one, as outlined in \Cref{treasuremap}.

In \Cref{sec:dmt-rank-three}, we implement the general program of \Cref{sec:interval-complex} in the rank-three case, proving that $K_W$ deformation retracts onto $X_W'$.
In \Cref{sec:consequences}, we deduce the following consequences: rank-three dual and standard Artin groups are isomorphic (\Cref{thm:standard-dual-isomorphism}), they satisfy the $K(\pi,1)$ conjecture (\Cref{thm:conjecture}), they are Garside groups (\Cref{thm:standard-garside}), they have a solvable word problem (\Cref{thm:word-problem}), and the non-spherical ones have a trivial center (\Cref{thm:center}).

\subsection*{Acknowledgements}
The authors thank the referees for their useful suggestions. Paolini acknowledges support from PRIN 2022A7L229 \textit{Algebraic and topological combinatorics} and INdAM's GNSAGA group. Salvetti acknowledges support from PRIN 2022S8SSW2 \textit{Algebraic and geometric aspects of Lie theory} and the MIUR Excellence Department Project awarded to the Department of Mathematics, University of Pisa, CUP I57G22000700001.

\section{Preliminaries}
\label{sec:preliminaries}

\subsection{The hyperbolic plane}
\label{sec:hyperbolic-plane}

Denote by $\H^2$ the abstract hyperbolic plane and by $\partial \H^2$ the space at infinity.
We refer to points in $\partial \H^2$ as \emph{points at infinity} or \emph{ideal points}.
Throughout this paper, we are going to use different models for the hyperbolic plane: the hyperboloid model, the half-plane model, the Poincaré model, and the Klein model.
We refer to \cite{cannon1997hyperbolic,benedetti1992lectures} for their definitions and the relations between them.
The hyperboloid and Klein models are directly linked to the geometry of hyperbolic Coxeter groups acting on the Tits cone (see \Cref{sec:coxeter-groups}).
However, we often find the half-plane model well-suited for explicit computations with isometries (as we see already in \Cref{sec:half-plane-model}).
Unless otherwise stated, the figures are drawn in the Poincaré model; both the half-plane and the Poincaré models are conformal (angles are preserved).

\subsubsection{Isometries of the hyperbolic plane}
Every isometry of the hyperbolic plane $\H^2$ can be written as a product of $3$ or fewer reflections %
and the \emph{reflection length} of an isometry is the minimal number of reflections needed.
Isometries can be classified as follows:
\begin{itemize}
    \item the identity (reflection length $0$),
    \item reflections with respect to a line (reflection length $1$),
    \item rotations around a point in $\H^2 \cup \partial \H^2$ (reflection length $2$),
    \item translations along a line (reflection length $2$) called the \emph{axis} of the translation (or the \emph{translation axis}),
    \item glide reflections, i.e., products of a reflection and a translation along the same line (reflection length $3$) called the \emph{axis} of the glide reflection.
\end{itemize}

The identity, reflections, and rotations around points in $\H^2$ are \emph{elliptic} isometries (they fix at least one point in $\H^2$). Rotations around points in $\partial \H^2$ are \emph{parabolic} isometries (they fix exactly one point in $\partial \H^2$ and no point in $\H^2$).
Translations and glide reflections are \emph{hyperbolic} isometries (they fix exactly two points in $\partial \H^2$, namely the endpoints of their axes, and fix no point in $\H^2$).

For an elliptic or parabolic isometry $u$, denote by $\Fix(u) \subseteq \H^2 \cup \partial \H^2$ the set of its fixed points (also called its \emph{fixed set}).
If $u$ is a hyperbolic isometry, then its axis is the set of points $x \in \H^2$ that minimize the distance $d(x, u(x))$.
For this reason, the axis of $u$ is often also called its \emph{min-set} and we denote it by $\Min(u)$.
The distance $d(x, u(x))$ for any $x \in \Min(u)$ is called the \emph{translation length} of $u$.
The axis of a hyperbolic isometry is naturally oriented according to the translation direction.

\subsubsection{Isometries in the hyperboloid model}
\label{iso-hyperboloid}
Denote by
\[ L = \{(x_1, x_2, x_3) \mid x_1^2 + x_2^2 - x_3^2 = -1 \text{ and } x_3 > 0 \} \subseteq \R^3 \]
the hyperboloid model.
Isometries in the hyperboloid model are restrictions of isometries of $\R^3$ with respect to the quadratic form $Q(x_1, x_2, x_3) = x_1^2 + x_2^2 - x_3^2$.
Conversely, every isometry of the quadratic space $(\R^3, Q)$ that preserves $L$ (as a set) restricts to an isometry of $L$.
In particular, reflections of $L$ are restrictions of reflections of $(\R^3, Q)$ with respect to linear ($2$-dimensional) planes that intersect $L$.

Given an isometry in the hyperboloid model (or an isometry of the abstract hyperbolic plane $\H^2$), we refer to its \emph{spectral radius} as the spectral radius of its extension to an isometry of $(\R^3, Q)$.
The spectral radius of an elliptic or parabolic isometry is always $1$, whereas the spectral radius of a hyperbolic isometry is $e^\mu$ where $\mu$ is the translation length of $u$ \cite[Corollary 3.5]{mcmullen2002coxeter}.

The \emph{moved set} of an isometry $u$ of $(\R^3, Q)$ is defined as $\Mov(u) = \im(u - \id) \subseteq \R^3$.
We also define the moved set of an isometry of the hyperbolic plane (viewed as an isometry in the hyperboloid model) as the moved set of its extension to an isometry of $(\R^3, Q)$.
The reflection length of an isometry of $\H^2$ is equal to the dimension of its moved set \cite[Theorem 5.2]{mccammond2021factoring}.

\subsubsection{Isometries in the half-plane model}
\label{sec:half-plane-model}
It is convenient to write the half-plane model as $H = \{ z\in \C \mid \text{Im}(z) > 0 \}$, so that the group of isometries is given by $\PGL(2, \R)$ \cite[I.6.14]{bridson2013metric}. An element of $\PGL(2, \R)$ acts on $H$ as follows:
\[
\begin{mymatrix}{cc}
a & b \\ c & d
\end{mymatrix} (z) =
\begin{cases}
\displaystyle\frac{az + b}{cz + d} & \text{if the determinant is positive} \\[0.5cm]
\displaystyle\frac{a \bar z + b}{c \bar z + d} & \text{if the determinant is negative}.
\end{cases}
\]
If we denote by $i\R_+$ the positive imaginary line in $H$, the matrices
\begin{equation}
    \begin{mymatrix}{cc}
    1 & 0\\
    0 & -1
    \end{mymatrix},
    \quad\quad
    \begin{mymatrix}{cc}
    \lambda & 0\\
    0 & 1
    \end{mymatrix} \quad 
    \text{for $\lambda > 1$}
    \label{eq:reflection-translation}
\end{equation}
correspond to the reflection with respect to $i\R_+$ and to the translations in the positive direction of the axis $i\R_+$, respectively.
Every reflection and every translation in $H$ is conjugate in $\PGL(2,\R)$ to one of the matrices above. In particular, if we consider the line $l$ in $H$ 
that is represented by a Euclidean semicircle centered in $\alpha\in \R$ and with radius $\rho$, 
the reflection with respect to $l$ is given by the matrix
\begin{equation}
    \begin{mymatrix}{cc}
    \alpha & \rho^2-\alpha^2 \\
    1 & -\alpha
    \end{mymatrix}
    \label{eq:reflection}
\end{equation}
(this can be computed by conjugating the reflection with respect to $i\R_+$ with an isometry that sends $i\R_+$ to $l$).
By multiplying the two matrices in \eqref{eq:reflection-translation}, we see that every glide reflection is conjugate to
\begin{equation}
    \begin{mymatrix}{cc}
    \lambda & 0\\
    0 & -1
    \end{mymatrix}
    \label{eq:glide-reflection}
\end{equation}
for some $\lambda > 1$.
If $u$ is a translation as in \eqref{eq:reflection-translation} or a glide reflection as in \eqref{eq:glide-reflection}, then its translation length is $\log \lambda$ and its spectral radius is $\lambda$.
More generally, the spectral radius of an isometry given as a matrix $M \in \PGL(2, \R)$ is equal to the absolute value of the ratio between the higher and the lower eigenvalue of $M$.

For later reference, we now prove a general lemma about the composition of glide reflections and reflections.

\begin{lemma}
	Let $w$ be a glide reflection and let $r$ be a reflection whose fixed line does not meet the axis of $w$ (not even at infinity).
	Then $wr$ is a translation and its (oriented) translation axis meets the (oriented) axis of $w$ with an angle less than $\frac{\pi}{2}$.
	In addition, the (oriented) translation axis of $wr$ intersects $\Fix(r)$ before the axis of $w$.
	The same applies to $rw$, except that the (oriented) translation axis of $rw$ intersects $\Fix(r)$ after the axis of $w$.
	\label{lemma:translations}
\end{lemma}

\begin{proof}
Assume without loss of generality that $w$ is of the form \eqref{eq:glide-reflection} and write $r$ in the form \eqref{eq:reflection}.
Since $\Fix(r)$ does not intersect the imaginary axis, we have $0<\rho < \vert\alpha\vert$.

The matrix associated with the composition $wr$ is
\[
\begin{mymatrix}{cc}
	\lambda\alpha & \lambda(\rho^2 - \alpha^2) \\
	-1 & \alpha
\end{mymatrix}.
\]
Every point $z \in \Fix(wr)$ satisfies $z^2 + (\lambda-1)\alpha z + \lambda(\rho^2 - \alpha^2) = 0$. This equation has two distinct real roots $z_1,z_2$ because $(\lambda+1) |\alpha| > 2 \sqrt{\lambda} \rho$ (recall that $\vert \alpha \vert >\rho$, and $\lambda+1 > 2 \sqrt{\lambda}$ for all $\lambda> 1$). Thus $wr$ is a translation whose axis is represented by a Euclidean circle $\gamma$ that meets the real line at $z_1$ and $z_2$. Since $z_1z_2<0$, we have that $\gamma$ meets the (upwardly oriented) imaginary axis.  After relabeling, we can assume that the translation axis is oriented from $z_1$ to $z_2$.
The situation is depicted in \Cref{fig:lemma-translations}.

Now, $(z_1 + z_2) / 2$ (the center of $\gamma$) has the same sign as $-\alpha$, which has the same sign as $\lambda(\rho^2 - \alpha^2) / \alpha$ (the image of $0$ under $wr$).
This has two consequences. First, $0$ is between $z_1$ and the center of $\gamma$. Thus, the (oriented) axis of $w$ and the (oriented) axis of $wr$ intersect at an angle less than $\frac{\pi}{2}$. Second, the sign of $z_1$ is the same as the sign of $\alpha$. Thus, if the axis of $wr$ intersects $\Fix(r)$, it does so before intersecting the axis of $w$.
It remains to show that $\gamma$ intersects $\Fix(r)$. For this, we show that exactly one of $z_1, z_2$ is between $\alpha-\rho$ and $\alpha + \rho$:
\[ (z_1 - (\alpha + \rho))(z_1 - (\alpha-\rho))(z_2 - (\alpha + \rho))(z_2 - (\alpha-\rho)) = (\rho^2 - \alpha^2)(\lambda+1)^2\rho^2 < 0. \]
The claim about $rw$ follows with an analogous computation, the difference being that $z_1+z_2$ has the same sign as $\alpha$.
\end{proof}

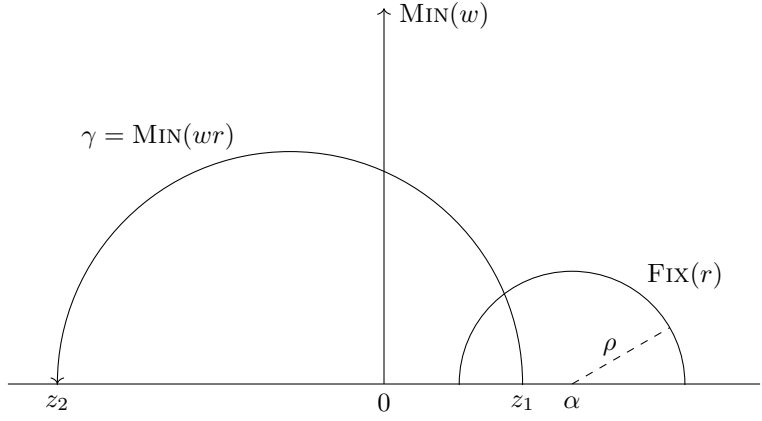
\begin{figure}
    \centering
    \begin{tikzpicture}
        \draw (-5, 0) -- (5, 0);
        
        \draw[->] (0, 0) -- (0, 5);
        
        \begin{scope}
            \clip (-5, 0) rectangle (5, 5);
            \draw (2.5, 0) circle (1.5);
        \end{scope}
        
        \begin{scope}
            \clip (-5, 0) rectangle (5, 5);
            \draw[
            decoration={markings, mark=at position 0.5 with {\arrow{>}}},
            postaction={decorate}
            ] (-1.25, 0) circle (3.09233);
        \end{scope}
        
        \draw[dashed] (2.5, 0) -- ($ (2.5, 0) + (1.299038106, 0.75) $);
        
        \node at (0, -0.25) {$0$};
        \node at (2.5, -0.25) {$\alpha$};
        \node at (3, 0.5) {$\rho$};
        \node at (-4.34233, -0.25) {$z_2$};
        \node at (1.84233, -0.25) {$z_1$};

        \node at (0.8, 4.9) {$\Min(w)$};
        \node at (4, 1.45) {$\Fix(r)$};
        \node at (-3, 3.3) {$\gamma = \Min(wr)$};
    \end{tikzpicture}
    \caption{Proof of \Cref{lemma:translations} (in the half-plane model).}
    \label{fig:lemma-translations}
\end{figure}

\begin{remark}
    In the setup of \Cref{lemma:translations}, suppose that $\Fix(r)$ meets the axis of $w$ in $0$ (a point at infinity). A similar analysis as in the previous proof (with  $|\alpha| = \rho$) yields that $wr$ is again a translation, but the translation axis of $wr$ meets the axis of $w$ in $0$ (with an angle of $0$).
    Both axes are oriented outwards from $0$.
\end{remark}

\subsection{Combinatorial topology}

A recurring theme of this paper is the relationship between the combinatorial properties of partially ordered sets (in turn related to groups via Garside theory, see \Cref{sec:garside-structures}) and the topological features of certain canonically associated spaces.

\subsubsection{Generalities on posets} 
We review some basic terminology, mainly following \cite{stanley2012enumerative}.
Let $(P,\leq)$ be a partially ordered set (a {\em poset}). We call $P$ {\em bounded} if it has a unique maximal element and a unique minimal element. Any two elements $p,q\in P$ define an {\em interval} $[p,q]:=\{u\in P \mid p\leq u \leq q\}$. Write $p\lessdot q$ if $[p,q]$ has cardinality $2$; in this case we say that $[p,q]$ is a poset cover. The {\em Hasse diagram} of $P$ is the graph whose vertices are the elements of $P$ and whose edges $E(P)$ are all poset covers in $P$.

A {\em chain} in $P$ is any totally ordered subset, i.e., any $C\subseteq P$ such that $c\leq c'$ or $c\geq c'$ for all $c,c'\in C$. The length of a chain is defined to be one less than its cardinality.
The length of the poset $P$ is the supremum of the lengths of all chains in $P$. The poset $P$ is {\em chain-finite} if every chain in $P$ has a finite length.
A chain-finite poset $P$ is called {\em graded} if any two maximal chains within the same interval have equal length. Equivalently, $P$ admits a {\em rank function} $\rk:P\to \mathbb Z$ such that $p\lessdot q$ implies $\rk(q)-\rk(p)=1$.
 
A poset $P$ is a {\em lattice} if every pair of elements $p_1,p_2\in P$ has a unique maximal lower bound and a unique minimal upper bound. 
Note that every chain-finite lattice is bounded.

The {\em order complex} of a partially ordered set $P$ is the abstract simplicial complex of all chains in $P$.  Topological properties of posets are defined by referring to the topology of its order complex. (Recall that every abstract simplicial complex has a geometric realization as a geometric simplicial complex that is unique up to homeomorphism.) Every chain determines a simplex of dimension equal to the chain's length. Thus, the order complex of a chain-finite bounded poset has a finite dimension.

\subsubsection{EL-labelings}
\label{sec:EL}
An edge-labeling of a poset $P$ is a function $\lambda: E(P)\to \Lambda$, where $\Lambda$ is any totally ordered set.  Given such a labeling, every finite, saturated chain  $p_1\lessdot\ldots\lessdot p_k$ in $P$ is associated with a $\Lambda$-word
\[
\lambda(p_1, p_2)
\lambda(p_2, p_3)
\cdots
\lambda(p_{k-1}, p_k).
\]
Using the ordering of $\Lambda$, chains of $P$ can be compared using the lexicographical order of the associated words.
A chain is called {\em increasing} if the associated word is strictly increasing.

An edge-labeling of a bounded poset $P$ is called an {\em EL-labeling} if every interval of $P$ has a unique increasing maximal chain and this chain lexicographically precedes all other maximal chains of the interval.
This notion was introduced by Bj\"orner \cite{bjorner1980shellable, bjorner1983lexicographically} in view of its strong topological implications for the order complex of $P$.
However, we use EL-labelings mainly as a combinatorial tool.

\subsubsection{Acyclic matchings of posets and discrete Morse theory}
\label{sec:dmt}
We briefly review Forman's discrete Morse theory for CW complexes \cite{forman1998morse} in the poset-theoretical formulation introduced by Chari \cite{chari2000discrete} and later extended by Batzies \cite{batzies2002discrete} to the case of infinite complexes. 

A matching of a poset $P$ is a matching of the associated Hasse diagram, i.e., a subset $\M\subseteq E(P)$ such that $m\cap m'=\varnothing$ whenever $m,m'\in \M$. An element of $P$ is {\em critical} with respect to a matching $\M$ if it is not contained in any $m\in \M$.
A matching defines an orientation of the edges of the Hasse diagram of $P$: an edge $[p,q]\in E(P)$ is oriented from $p$ to $q$ if $[p,q]\in \M$ and from $q$ to $p$ otherwise. The matching $\M$ is {\em acyclic} if the resulting oriented graph $H_\M$ has no directed cycles.

If $P$ is a graded poset, then any directed cycle of $H_\M$ must alternate between edges in $\M$ and outside of $\M$.
Indeed, if $p\to q$ is an oriented edge and $\rk$ is a rank function for $P$, then $\rk(q)-\rk(p)$ is equal to $1$ if the edge is in $\M$ and $-1$ otherwise.
Thus a matching $\M$ of a graded poset $P$ is acyclic if and only if the Hasse diagram of $P$ has no closed cycles that alternate between edges in $\M$ and outside of $\M$.

An acyclic matching $\M$ is called {\em proper} if, for every $p\in P$, there are only finitely many $q\in P$ that can be reached from $p$ by a directed path in $H_\M$.

Let $X$ be a finite-dimensional $CW$-complex. The {\em poset of cells} $\F(X)$  is the set of all (open) cells of $X$ with the partial order given by inclusion of closure: $\tau\leq \sigma$ if $\overline{\tau}\subseteq \overline{\sigma}$.
Recall that every cell $\sigma\subseteq X$ has a characteristic map $\Phi_\sigma: D^n\to X$ where $D^n$ is the closed $n$-ball and $n=\dim(\sigma)$. The poset $\F(X)$ is chain-finite and graded with rank function given by  $\rk(\sigma):=\dim(\sigma)$.  

Let $\sigma\in \F(X)$ be a cell of dimension $n$. A {\em regular face} of $\sigma$ is any cell $\tau\lessdot \sigma$ such that  $\Phi_\sigma$ restricts to a homeomorphism $\Phi^{-1}(\tau) \to \tau$, and  $\overline{\Phi_\sigma^{-1}(\tau)}$ is homeomorphic to an $(n-1)$-ball in $D^n$. A matching $\M$ of $\F(X)$ is called {\em regular} if $[\tau,\sigma]\in \M$ implies that $\tau$ is a regular face of $\sigma$.

\begin{theorem}[\cite{forman1998morse, chari2000discrete, batzies2002discrete}]
    Let $X$ be a finite-dimensional CW complex and let $\M$ be a proper and regular acyclic matching of $\F(X)$.
    Suppose that the set of critical elements of $\M$ forms a subcomplex $Y$ of $X$. Then $X$ deformation retracts onto $Y$ (we write $X \searrow Y$). In particular, the inclusion $Y\hookrightarrow X$ is a homotopy equivalence.
    \label{thm:dmt}
\end{theorem}

A useful tool for constructing acyclic matchings is given by the following well-known theorem.

\begin{theorem}[Patchwork theorem {\cite[Theorem 11.10]{kozlov2007combinatorial}}]
	Let $\eta \colon P \to Q$ be a poset map.
	For all $q \in Q$, assume there is an acyclic matching $\M_q \subseteq E(P)$ that involves only elements of the fiber $\eta^{-1}(q) \subseteq P$.
	Then the union of these matchings is an acyclic matching on $P$.
	\label{thm:patchwork}
\end{theorem}

\subsection{Interval groups and Garside structures}
\label{sec:garside-structures}

Let $G$ be a group with a generating set $R$ not containing the identity of $G$ and such that $R = R^{-1}$.
Then $G$ is partially ordered by assigning $x \leq y$ whenever $l(x) + l(x^{-1}y) = l(y)$, where $l \colon G \to \mathbb{N}$ denotes the length function induced by $R$.
In other words, the relation $x \leq y$ holds if and only if there is a geodesic path between $1$ and $y$ passing through $x$, inside the right Cayley graph of $G$ (with respect to the generating set $R$).

Fix an element $g \in G$ and consider the interval $[1,g] \subseteq G$, consisting (by definition) of all elements $x \in G$ such that $l(x) + l(x^{-1}g) = l(g)$.
The interval $[1,g]$ is \emph{balanced} if the set of elements $x\in G$ satisfying $l(x) + l(x^{-1}g) = l(g)$ coincides with the set of elements $x\in G$ satisfying $l(gx^{-1}) + l(x) = l(g)$.
In other words, this condition requires that the interval $[1,g]$ inside the right Cayley graph (as we have defined it above) contains the same elements as the interval $[1,g]$ inside the left Cayley graph.

Assuming to have a balanced interval $[1,g]$, construct a new group $G_g$ (called an \emph{interval group}) and a CW complex $K$ (called an \emph{interval complex}) as follows.
The interval complex $K$ is a $\Delta$-complex (in the sense of \cite{hatcher}) having one $d$-dimensional simplex $\sigma$ denoted by $[x_1|x_2|\dotsb|x_d]$ for every $d$-tuple of elements $x_1, x_2, \dotsc, x_d \in [1,g] \setminus \{1\}$ such that $l(x_1) + l(x_2) + \dotsb + l(x_d) = l(x_1x_2\dotsm x_d)$ and $x_1x_2\dotsm x_d \in [1,g]$.
The faces of $[x_1|x_2|\dotsb|x_d]$ are given by $\partial_0(\sigma) = [x_2|\dotsb|x_d]$,
$\partial_{i}(\sigma) = [x_1|\dotsb|x_ix_{i+1}|\dotsb|x_d]$ for $i=1, \dotsc, d-1$, and
$\partial_d(\sigma) = [x_1|\dotsb|x_{d-1}]$.
See \cite[Definition 2.8]{paolini2021proof} for more details.
Note that $K$ is a quotient of the order complex of $[1,g]$.
The $0$-cell $[\,]$ of $K$ is not a regular face of the $1$-cells; all other faces are regular.

Define the interval group $G_g$ as the fundamental group of $K$.
Then, $G_g$ has a presentation with a generator for each $1$-cell $[x]$ of $K$ and a relation $[x][y] = [xy]$ for each $2$-cell $[x|y]$ of $K$.
The main reason why we consider interval groups and interval complexes is the following result.

\begin{theorem}[\cite{bestvina1999,dehornoy1999gaussian,charney2004bestvina,mccammond2005introduction,dehornoy2015foundations}]
    If the interval $[1,g]$ is a balanced lattice, then $G_g$ is a \emph{Garside group} and the interval complex $K$ is a classifying space for $G_g$.
    In addition, the word problem for $G_g$ is solvable, provided that one can algorithmically check equality and compute meets and joins in $[1,g]$.
    \label{thm:garside}
\end{theorem}

Garside groups were introduced by Dehornoy and Paris \cite{dehornoy1999gaussian}, building on previous work of Garside \cite{garside1969braid}.
We will not need the actual definition of a Garside group.

\subsection{Coxeter groups}
\label{sec:coxeter-groups}

In this section, we outline some basics of the general theory of Coxeter groups. We refer to \cite{bourbaki1968elements, humphreys1992reflection} for a more thorough treatment.

A {\em Coxeter system} is a group $W$ with a distinguished set of generators $S$ such that
\begin{equation}
	\label{eq:coxeter-presentation}
 	W = \< S \mid (st)^{m(s,t)} = 1 \; \text{for all } s,t \in S \text{ such that $m(s,t) \neq \infty$} \>,
\end{equation}
for some function $m:S\times S\to \mathbb N\cup\{\infty\}$ satisfying $m(s,s')=1$ if $s=s'$ and $m(s,s')\geq 2$ otherwise. The function $m$ can be encoded in a {\em Coxeter diagram}, i.e., a graph on the vertex set $S$ where two vertices $s,s'$ are joined by an edge exactly when $m(s,s')\geq 3$. The edge is labeled with $m(s,s')$ when $m(s,s')\geq 4$.
The group $W$ is called a {\em Coxeter group}.
For us, a Coxeter group $W$ always implicitly carries with it a fixed generating set $S$ which makes $(W, S)$ a Coxeter system.
For this reason, we often speak of properties of a Coxeter group $W$ which depend on the Coxeter system (and not only on the group structure).
For example, a Coxeter group $W$ is called \emph{irreducible} if the corresponding Coxeter diagram is connected.

The {\em parabolic subgroup} of $W$ associated with a subset $I\subseteq S$ is the subgroup $W_I$ of $W$ that is generated by the elements of $I$.
Following \cite{mcmullen2002coxeter}, we call an element $u\in W$  {\em essential} if it is not conjugated into any parabolic subgroup $W_I$ with $I\neq S$.
Let $n:=\vert S\vert$. This quantity is the {\em rank} of the given Coxeter group.

\subsubsection{Geometric representation and reflections}
\label{sec:geometric-representation}
The function $m$ determines a symmetric bilinear form $B$ on the vector space $\R^S$ defined on basis vectors as
\[
B(e_s,e_{s'}):= -\cos(\pi/m(s,s'))
\]
where we set $B(e_s,e_{s'})=-1$ if $m(s,s')=\infty$.
Now, to every $s\in S$ is naturally associated the {\em reflection}
\[ r_s: x\mapsto x - 2B(x,e_s) \, e_s. \]
The assignment $s \mapsto r_s$ extends to a linear action of $W$ on $\mathbb R^S$ that preserves the bilinear form $B$ (see \cite[5.3]{humphreys1992reflection}).
Denote by $R$ the set of all elements of $W$ that act as reflections (i.e., that fix a hyperplane and send some non-zero vector to its opposite). This is called the \emph{set of reflections} of the Coxeter group and coincides with the set of all conjugates of elements of $S$. In particular, $R$ generates $W$. The {\em absolute length} (also called \textit{reflection length}) of an element $u\in W$ is the minimum length $l_R(u)$ of a reduced expression of $u$ as a word in the generators $R$.
The elements of $S$ are called \emph{simple reflections}.

\subsubsection{Coxeter elements and dual Coxeter systems}
A {\em Coxeter element} of a Coxeter system $(W,S)$ is any product of all elements of $S$ in some order.
Every Coxeter element has reflection length equal to $n = |S|$ (see \cite[Lemma 3.8]{ingalls2009noncrossing} and \cite[Lemma 5.1]{paolini2021proof}) and is essential \cite{paris2007irreducible}.
For any choice of a Coxeter element $w$ of $(W,S)$, the triple $(W,R,w)$ is often called a {\em dual Coxeter system}.
Associated with any dual Coxeter system is the corresponding poset of {\em noncrossing partitions},\footnote{The name refers to a well-known combinatorial interpretation of such elements in the case of the symmetric group, see for example \cite{armstrong2009generalized}.} namely
\[
[1,w]:=\{u\in W \mid l_R(u) + l_R(u^{-1}w) = l_R(w)\},
\]
partially ordered by
\[
u \leq v \textrm{ if and only if }
l_R(u) + l_R({u}^{-1}v) = l_R(v).
\]
As the notation suggests, the poset $[1,w]$ coincides with the interval between the identity $1$ and the Coxeter element $w$ in the right Cayley graph of $W$ with respect to the generating set $R$.
In particular, it contains $1$ and $w$ as its unique minimal and maximal elements.
Note that $[1,w]$ is balanced because the generating set $R$ is closed under conjugation.
The noncrossing partition posets $[1,w]$ are strictly related to  ``dual Garside structures'', see \Cref{sec:garside-structures}.

\subsubsection{Spherical, affine, and hyperbolic Coxeter groups}
\label{sec:spherical-affine-hyperbolic-coxeter}

A Coxeter group $W$ is called \emph{spherical} if the bilinear form $B$ (defined in \Cref{sec:geometric-representation}) is positive definite; it is called \emph{affine} if $B$ is positive semidefinite but not positive definite.
The spherical case occurs precisely when $W$ is finite.

For simplicity, suppose now that $B$ is nondegenerate. Then $B$ induces a pairing, and hence an identification, between $\R^S$ and its dual $V$.
In particular, we have an induced action of $W$ on $V$.
The generators $s \in S$ act as reflections with respect to the linear hyperplanes $\widehat{H}_s:=\{p\in V \mid B(p,e_s)=0\}$ that bound a closed polyhedron $D\subseteq V$ (a simplicial cone).
This polyhedron is a fundamental region for the action of $W$ on $V$ and the union of all $W$-translates of $D$ is a convex cone $I$ in $V$, the {\em Tits cone} \cite[Section 5.13]{humphreys1992reflection}. 
The cone $I$ is tiled by copies of $D$. This tiling corresponds to the subdivision of $I$ determined by the collection of all reflection hyperplanes $\widehat{H}_r$ for $r \in R$.

The Coxeter group $W$ is called {\em hyperbolic} if the bilinear form $B$ is nondegenerate of signature $(n-1,1)$ and every vector $p$ in the Tits cone satisfies $B(p, p) < 0$ \cite[Section 6.8]{humphreys1992reflection}.
If $W$ is hyperbolic, then the action of $W$ on $V$ is determined by its restriction to the hyperboloid model of the hyperbolic space $\H^{n-1}$ given by all $p$ with $B(p,p)=-1$, with the metric induced by $B$.
Since the action of $W$ preserves the metric, $W$ acts by isometries of the hyperbolic space. Thus, hyperbolic Coxeter groups of rank $n$ are a subclass of all discrete groups of isometries of $\H^{n-1}$ generated by reflections. Note that this class is strictly larger; see \cite{vinberg1971discrete}.

Let $W$ be a hyperbolic Coxeter group and, for all $r\in R$, let $H_r$ be the intersection of $\widehat{H}_r$ with the hyperboloid.
The set $\{H_r\}_{r\in R}$ is the {\em reflection arrangement} of $W$. It is a locally finite set of hyperplanes of $\H^{n-1}$. The open cells of the induced subdivision of the hyperbolic space are the (open) {\em chambers} of the reflection arrangement. Chambers are naturally and bijectively labeled by elements of $W$, once we label the interior of $D$ with the identity element. 

The Coxeter graphs corresponding to spherical, affine, and hyperbolic Coxeter systems have been completely classified (see \cite[Chapter 6]{humphreys1992reflection}). The construction of the reflection arrangement can be extended to all cases (see e.g.\ \Cref{fig:hyperbolic-arrangements,fig:spherical-affine-arrangements}).

\subsection{Standard and dual Artin groups}
\label{sec:artin-groups}

For any Coxeter system $(W, S)$, where $W$ is presented as in \eqref{eq:coxeter-presentation}, there is an associated Artin group defined as
\begin{equation}
    \label{eq:artin-presentation}
 	G_W = \< S \mid \!\!\underbrace{stst\dotsm}_{m(s,t) \text{ terms}} \! = \! \underbrace{tsts\dotsm}_{m(s,t) \text{ terms}} \forall\, s,t \in S \text{ such that $m(s,t) \neq \infty$} \>.
\end{equation}
An Artin group is called irreducible, spherical, affine, or hyperbolic if the corresponding Coxeter group is respectively irreducible, spherical (i.e., finite), affine, or hyperbolic.

Define the \emph{configuration space} associated with $W$ as
\[ Y = (I \times I) \setminus \bigcup_{r \in R} \widehat H_r \times \widehat H_r. \]
Then $W$ acts freely and properly discontinuously on $Y$, and the quotient space $Y_W = Y/W$ is the \emph{orbit configuration space} associated with $W$.
The fundamental group of $Y_W$ is isomorphic to the Artin group $G_W$ \cite{van1983homotopy}.
The orbit configuration space $Y_W$ has the homotopy type of a CW complex (called the \emph{Salvetti complex}) having a $k$-cell for each subset $T \subseteq S$ of cardinality $k$ that generates a finite subgroup $W_T \subseteq W$ \cite{salvetti1987topology, salvetti1994homotopy} (see also \cite{paris2012k}).

In full generality, Artin groups are not well understood, in the sense that there are very few known results that apply to all of them.
The following are among the most important open problems on Artin groups: the word problem; determining the center; solving the $K(\pi, 1)$ conjecture, due to Brieskorn, Arnol'd, Pham, and Thom.

\begin{conjecture}[$K(\pi, 1)$ conjecture]
    The orbit configuration space $Y_W$ is a classifying space for the corresponding Artin group $G_W$.
\end{conjecture}

Among other things, the $K(\pi, 1)$ conjecture gives a way to compute the homology and cohomology of Artin groups and implies that Artin groups are torsion-free (a property that is also unknown in general).
So far, the $K(\pi, 1)$ conjecture has been proved for spherical Artin groups (by Deligne \cite{deligne1972immeubles}), affine Artin groups (by the second and third authors \cite{paolini2021proof}), $2$-dimensional and FC-type Artin groups (by Charney and Davis \cite{charney1995k}).
Special cases of these were previously proved by Fox and Neuwirth \cite{fox1962braid}, Brieskorn \cite{brieskorn1973groupes}, Okonek \cite{okonek1979dask}, Hendriks \cite{hendriks1985hyperplane}, Callegaro, Moroni and Salvetti \cite{callegaro2010k}.
See \cite{paris2012k} for a survey on this problem (written before the full solution of the affine case) and \cite{paolini2021dual} for an exposition of the ``dual approach'' introduced to solve the affine case, which we use and extend here.
See also \cite{charney1995k, godelle2012basic} for an overview of open problems on Artin groups (mostly up to date, except for the affine case).

If $W$ is a finite Coxeter group, then $G_W$ is the interval group associated with the interval $[1,\delta] \subseteq W$, where we use $S$ as the generating set of $W$ and define $\delta$ as the longest element of $W$.
The interval $[1, \delta]$ is a balanced lattice and thus makes $G_W$ a Garside group (by \Cref{thm:garside}). This is known as the \emph{standard Garside structure} on spherical Artin groups, introduced by Garside for braid groups \cite{garside1969braid} and developed and studied by Brieskorn, Saito, Deligne, and others \cite{brieskorn1972artin, deligne1972immeubles}.
The interval complex associated with $[1,\delta]$ is homotopy equivalent to the orbit configuration space $Y_W$, in accordance with the $K(\pi,1)$ conjecture (this was explicitly proved in \cite{delucchi2009combinatorics}).

An alternative way to realize spherical Artin groups as Garside groups was introduced by Bessis \cite{bessis2003dual}, following prior work of Birman, Ko, and Lee on braid groups \cite{birman1998new}.
Consider a noncrossing partition poset $[1,w]$ in a finite Coxeter group $W$.
It turns out that $[1,w]$ is a lattice, and the corresponding interval group is naturally isomorphic to the Artin group $G_W$.
This is known as the \emph{dual Garside structure} on spherical Artin groups.
The interval complex associated with $[1,w]$ is therefore another model for the classifying space of the spherical Artin group $G_W$, with a combinatorial structure that is substantially different from the Salvetti complex $X_W$ and the interval complex arising from the standard Garside structure.

Noncrossing partition posets in general Coxeter groups are not always lattices.
Indeed, McCammond showed that the lattice property fails in most affine cases \cite{mccammond2015dual}.
Nevertheless, the dual structure (which is not necessarily Garside) proved useful to answer the most important open questions on affine Artin groups, such as the word problem, the center, and the $K(\pi, 1)$ conjecture \cite{mccammond2017artin, paolini2021proof}.
For all affine Coxeter groups $W$, McCammond and Sulway proved (among other things) that the interval group $W_w$ associated with $[1,w]$ is isomorphic to the usual Artin group $G_W$ \cite{mccammond2017artin}; the second and third author proved that the interval complex $K_W$ associated with $[1,w]$ is a classifying space and deformation retracts onto a finite subcomplex $X_W' \subseteq K_W$ which is homotopy equivalent to the orbit configuration space $Y_W$ \cite{paolini2021proof}.

The subcomplex $X_W'$ can be defined for an arbitrary Coxeter system and is always homotopy equivalent to $Y_W$. We review the definition of $X_W'$ in \Cref{sec:interval-complex}.
Therefore, proving that $K_W$ deformation retracts onto $X_W'$ implies the isomorphism $G_W \cong W_w$ between the standard and dual Artin groups.
In addition, if the interval $[1,w]$ is a lattice, then a deformation retraction $K_W \cong X_W'$ yields a Garside structure on $G_W$ and implies the $K(\pi, 1)$ conjecture.
These implications are discussed more thoroughly in \cite{paolini2021dual}.

\section{Coxeter elements in rank-three hyperbolic groups}
\label{sec:coxeter-elements}

Let $(W,S)$ be a rank-three Coxeter system.
Recall from \cite[Section 6.7]{humphreys1992reflection} that, if the three labels $m(s,t)$ for $s \neq t$ are denoted by $m_1, m_2, m_3$, then $W$ is hyperbolic if and only if $\frac {1}{m_1} + \frac {1}{m_2} + \frac {1}{m_3} < 1$ (set $\frac{1}{m_i} = 0$ if $m_i = \infty$).
In particular, all irreducible rank-three Coxeter systems are hyperbolic except when the triple $(m_1, m_2, m_3)$ takes the values $(2, 3, 3), (2, 3, 4), (2, 3, 5), (2, 3, 6), (2, 4, 4)$, or $(3, 3, 3)$, up to permutations.
These special cases are either spherical or affine and the corresponding arrangements are shown in \Cref{fig:spherical-affine-arrangements}.
All questions we consider have been solved already in the spherical and affine cases,
so from now on, \textbf{assume that the Coxeter system $(W, S)$ is hyperbolic}, thus acting by isometries on the hyperbolic plane $\H^2$. Two examples are shown in \Cref{fig:hyperbolic-arrangements}.

\subsection{Coxeter elements and their axes}
Let $S=\{a,b,c\}$ and fix a Coxeter element $w = abc$.
Note that $w$ has reflection length equal to $3$ (both in $W$ and in the group of isometries of $\H^2$), so it is a glide reflection (see \Cref{sec:hyperbolic-plane}).

\begin{definition}
    The (oriented) axis of $w$ is called the \emph{Coxeter axis} and denoted by $\ell$.
    An \emph{axial chamber} is a chamber whose interior intersects the Coxeter axis. We think of the orientation of the Coxeter axis as defining the ``positive'' or ``upward'' direction. Accordingly, a point $p \in \ell$ (or an axial chamber $C$) is {\em above} another point $p' \in \ell$ (or another axial chamber $C'$) if it is further along the Coxeter axis in the positive direction.
    Any vertex of an axial chamber is called an \emph{axial vertex}.
    See \Cref{fig:coxeter-axis}.
    \label{def:coxeter-axis}
\end{definition}

\begin{figure}
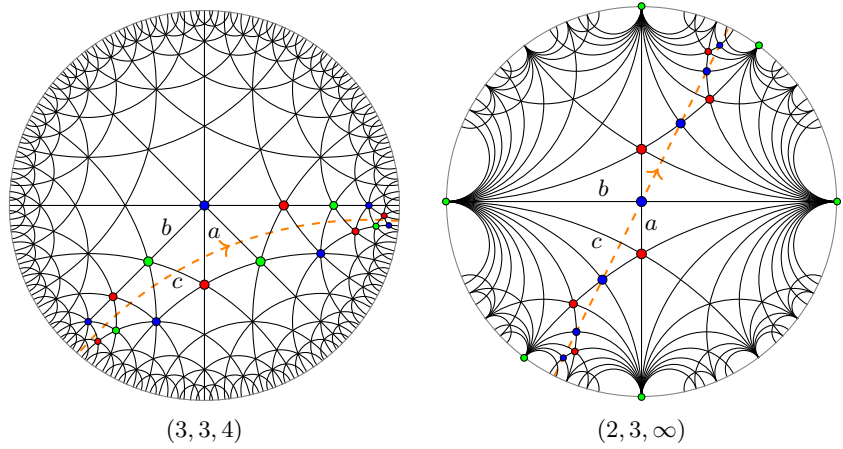

    \newcommand{\scale}{2.6}
    \begin{center}
        \begin{subfigure}[t]{.45\linewidth}
            \centering
            \begin{tikzpicture}[scale=\scale]
                \clip (0,0) circle (1);
                
                \draw[
                    white,
                    decoration={markings, mark=at position 0.52 with {\arrow[orange, thick]{>}}},
                    postaction={decorate}
                ] (-142.629284:1) -- (9.629284:1);
                
                \begin{scope}[every path/.style={thick,dashed,orange}]
                    \hgline{-130.429284}{-4.570716}{}
                \end{scope}
                
                \input{lines334}
                
                \input{axialvertices334}
                
                \draw[black!50, thick] (0,0) circle (1);
                
                \node at (0.05, -0.14) {\small $a$};
                \node at (-0.195, -0.11) {\small $b$};
                \node at (-0.14, -0.395) {\small $c$};
            \end{tikzpicture}
            \caption*{$(3, 3, 4)$}
        \end{subfigure}
        \begin{subfigure}[t]{.45\linewidth}
            \centering
            \begin{tikzpicture}[scale=\scale]
                \begin{scope}
    	            \clip (0,0) circle (1);

                    \draw[
                        white,
                        decoration={markings, mark=at position 0.595 with {\arrow[orange, thick]{>}}},
                        postaction={decorate}
                    ] (-116.565051:1) -- (63.434949:1);

                    \begin{scope}[every path/.style={thick,dashed,orange}]
                        \hgline{-116.565051}{63.434949}{}
                    \end{scope}

    	            \input{lines23infty}

    	            \draw[black!50, thick] (0,0) circle (1);
	            \end{scope}

                \input{axialvertices23infty}

                \node at (0.05, -0.12) {\small $a$};
                \node at (-0.195, 0.07) {\small $b$};
                \node at (-0.23, -0.21) {\small $c$};
            \end{tikzpicture}
            \caption*{$(2, 3, \infty)$}
        \end{subfigure}
    \end{center}
    
    \caption{Reflection arrangements of \Cref{fig:hyperbolic-arrangements} with the axis of a Coxeter element $w = abc$ (orange dashed line).
    The axial vertices are colored based on their orbit under the action of the infinite cyclic group generated by $w$ (see \Cref{lemma:axial-vertices-orbits}).}
    \label{fig:coxeter-axis}
\end{figure}

In the hyperboloid model $L \subseteq \R^3$, the Coxeter axis $\ell$ is the intersection between $L$ and a linear plane $P \subseteq \R^3$ which we call the \emph{Coxeter plane}.
As noted in \cite[Section 2]{paolini2021dual}, $P$ is indeed the analog of the classical Coxeter plane of spherical Coxeter groups.

Let $C$ be any chamber of the reflection arrangement. In particular, the closure $\bar C$ intersects every $W$-orbit in exactly one point.
Denote by $\pi_C\colon \H^2 \to \bar C$ the projection defined by $\pi_C(x):=Wx\cap \bar C$ for every $x\in \H^2$.
The following lemma describes the location of the Coxeter axis $\ell$.

\begin{lemma}
    For every point $x \in \ell$, the image of the segment $[x, w(x)] \subseteq \ell$ under $\pi_C$ is the shortest loop that touches the three walls of $C$ and coincides with the orthic triangle of $C$ (see \Cref{fig:pedal-triangle}).
    \label{lemma:pedal-triangle}
\end{lemma}
\begin{proof}
    The orthic triangle of a triangle $C \subseteq \H^2$ is the (only) shortest loop that touches all three walls of $C$.
    This can be proved with the same argument as the one used in \cite[Chapter 6]{rademacher2018enjoyment} for the Euclidean case.

    By \cite[Theorem 4.1 and Proposition 4.2]{mcmullen2002coxeter}, there exists a Coxeter element $w' \in W$ such that $\pi_C([x', w'(x')])$ is the shortest loop that touches all three walls of $C$, where $x'$ is any point on the axis of $w'$.
    Note that the length of any such loop is equal to the translation length $\mu$ of $w'$.
    
    Every Coxeter element is conjugate in $W$ to one of $abc, bca, cab, cba, bac$, or $acb$ (where $\{a, b, c\}$ is a fixed set of simple reflections).
    The first three ($abc, bca$, and $cab$) can be obtained from each other via source-sink flips \cite[Lemma 7.4]{mccammond2015dual}, hence they are geometrically equivalent Coxeter elements and thus have the same translation length.
    The same holds for the last three ($cba, bac$, and $acb$). In addition, each of the first three is the inverse of one of the last three, so all Coxeter elements have translation lengths equal to $\mu$.
    In particular, the loop $\gamma = \pi_C([x, w(x)])$ has length $\mu$.
    Since $w$ is essential, \cite[Proposition 4.3]{mcmullen2002coxeter} ensures that $\gamma$ touches all three walls of $C$.
    Therefore, $\gamma$ is the shortest loop that touches the three walls of $C$.
\end{proof}

\begin{figure}
    \centering
    \begin{tikzpicture}[scale=3.6]
        \clip (0,0) circle (1);
        
        \draw[
            white,
            decoration={markings, mark=at position 0.52 with {\arrow[orange, thick]{>}}},
            postaction={decorate}
        ] (-142.629284:1) -- (9.629284:1);
        
        \draw[
            white,
            decoration={markings, mark=at position 0.7 with {\arrow[blue, thick]{>}}},
            postaction={decorate}
        ] (-30.47:1) -- (-183.23:1);
        
        \draw[
            white,
            decoration={markings, mark=at position 0.253 with {\arrow[purple, thick]{<}}},
            postaction={decorate}
        ] (-89.92:1) -- (113.77:1);
        
        \begin{scope}
            \clip (45:1) -- (225:1) -- (-45:1) -- cycle;
            \clip (90:1) -- (-90:1) -- (180:1) -- cycle;
            \clip (180:1) -- (317:1) -- (90:1) -- cycle;    %

            \begin{scope}[every path/.style={very thick,black!40}]
                \begin{scope}
                    \hgline{-130.429284}{-4.570716}{}
                \end{scope}
    
                \begin{scope}
                    \hgline{-49.570716}{-175.429284}{}
                \end{scope}
    
                \begin{scope}
                    \hgline{-94.570716}{139.570716}{}
            \end{scope}
            \end{scope}
        \end{scope}

        \begin{scope}[every path/.style={thick,dashed,orange}]
            \hgline{-130.429284}{-4.570716}{}
        \end{scope}
        
        \begin{scope}[every path/.style={thick,dashed,blue}]
            \hgline{-49.570716}{-175.429284}{}
        \end{scope}
    
        \begin{scope}[every path/.style={thick,dashed,purple}]
            \hgline{-94.570716}{139.570716}{}
        \end{scope}
        
        \input{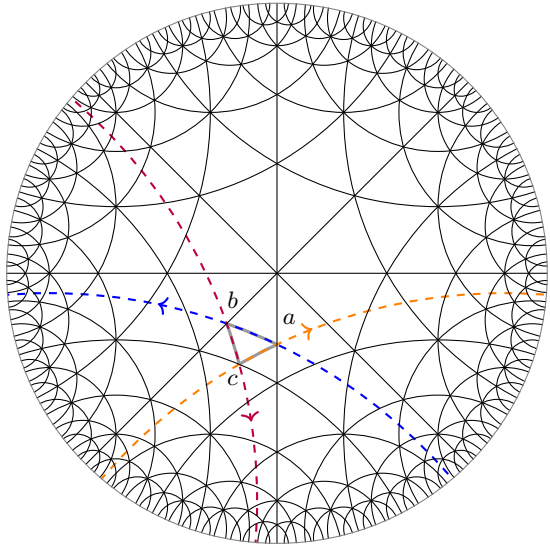}
        
        \draw[black!50, thick] (0,0) circle (1);
        
        \node at (0.046, -0.17) {\small $a$};
        \node at (-0.163, -0.096) {\small $b$};
        \node at (-0.165, -0.395) {\small $c$};
    \end{tikzpicture}
    \caption{Axes of the Coxeter elements $abc$ (in orange), $bca$ (in blue), and $cab$ (in purple).
    The orthic triangle of the chamber delimited by $\Fix(a), \Fix(b)$, and $\Fix(c)$ is highlighted in gray.}
    \label{fig:pedal-triangle}
\end{figure}

We now derive two useful corollaries: first, the Coxeter axis is not a reflection line (so axial chambers exist); second, every axial chamber induces a factorization of $w$.

\begin{corollary}
    The Coxeter axis $\ell$ does not coincide with any reflection line of $W$.
    In addition, if two reflection lines intersect $\ell$ in the same point, then they are perpendicular.
    \label{lemma:reflections-and-axis}
\end{corollary}

\begin{proof}
    The orthic triangle of any chamber $C$ touches every wall at exactly one point. In addition, it touches two walls at the same point if and only if the two walls are perpendicular (in this case, the orthic triangle is degenerate since two vertices coincide).
    The statement then follows from \Cref{lemma:pedal-triangle}, since $\pi_C(\ell)$ is the orthic triangle of $C$.
\end{proof}

\begin{corollary}
    Suppose that $C$ is an axial chamber, with associated reflections $s_1,s_2,s_3$ in the order the sides of $C$ are touched by the loop $\pi_C([x, w(x)])$ for any $x \in \ell \cap C$.
    Then $w = s_1s_2s_3$.
    \label{lemma:axial-factorization}
\end{corollary}

\begin{proof}
    Let $x \in C \cap \ell$.
    Now, $\pi_C(x)=\pi_C(w(x))$, and no other point in $\pi_C(\ell)$ is in the same $W$-orbit. By \Cref{lemma:pedal-triangle}, the segment of $[x, w(x)] \subseteq \ell$ intersects exactly three reflection lines, and the corresponding reflections are $s_1$, $s_1s_2s_1$, and $s_1s_2s_3s_2s_1$.
    Therefore $w = (s_1s_2s_3s_2s_1)(s_1s_2s_1) s_1 = s_1s_2s_3$.
\end{proof}

We close this section with a geometric observation on axial vertices which is exemplified in \Cref{fig:coxeter-axis}.

\begin{lemma}
    Let $p$ be an axial vertex.
    Every axial chamber has exactly one vertex in the set $\{w^j(p)\mid j \in \Z\}$.
    \label{lemma:axial-vertices-orbits}
\end{lemma}

\begin{proof}
    Let $C$ be an axial chamber.
    Since $\bar C$ is a fundamental domain for the action of $W$ on $\H^2$, at most one vertex of $C$ belongs to the orbit $\{w^j(p)\mid j \in \Z\}$.
    In particular, the three vertices of $C$ necessarily belong to three different orbits.
    
    Let $C'$ be another axial chamber.
    By induction on the number of axial chambers between $C$ and $C'$, we prove that the orbits of the three vertices of $C$ are the same as the orbits of the three vertices of $C'$.
    It is enough to consider the case where $C'$ is the axial chamber immediately above $C$.
    Denote by $s_1, s_2, s_3$ the reflections with respect to the walls of $C$, ordered as in \Cref{lemma:axial-factorization}, so that $w = s_1s_2s_3$.
    
    \emph{Case 1:} suppose that $C$ and $C'$ are separated by a single reflection line, $\Fix(s_1)$.
    Then $C' = s_1(C)$.
    The vertex $q'$ of $C'$ opposite to $\Fix(s_1)$ is equal to $s_1(q) = w(q)$ where $q$ is the vertex of $C$ opposite to $\Fix(s_1)$.
    The other two vertices are in common between $C$ and $C'$.
    See for example \Cref{fig:coxeter-axis}, left.
    
    \emph{Case 2:} suppose that $C$ and $C'$ are separated by two reflection lines, $\Fix(s_1)$ and $\Fix(s_2)$, which are orthogonal by \Cref{lemma:reflections-and-axis}.
    Then $C' = s_1s_2(C)$.
    The vertices $q_1'$ and $q_2'$ of $C'$ opposite to $\Fix(s_1)$ and $\Fix(s_2)$ are equal to $s_1s_2(q_1) = w(q_1)$ and $s_1s_2(q_2) = w(q_2)$, where $q_1$ and $q_2$ are the vertices of $C$ opposite to $\Fix(s_1)$ and $\Fix(s_2)$.
    The third vertex is in common between $C$ and $C'$.
        See for example \Cref{fig:coxeter-axis}, right.

    In all cases, the vertices of $C$ and the vertices of $C'$ belong to the same three orbits.
\end{proof}

\subsection{Reflections}

Let $r$ be a reflection in the noncrossing partition poset $[1, w]$.
By analogy with the affine case, we say that $r$ is \emph{vertical} if $\Fix(r)$ intersects the Coxeter axis $\ell$ and \emph{horizontal} if it does not.

\begin{remark}
    If $r$ is a horizontal reflection, then $\Fix(r)$ does not intersect $\ell$ at infinity, because the distance between $\Fix(r)$ and $\ell$ is bounded away from $0$.
    Indeed, we can fix a point $x \in \ell$ and find an $\epsilon$-neighborhood $N$ of the closed segment $[x, w(x)]$ which intersects no fixed lines of horizontal reflections; then the union of all $w^k(N)$ for $k \in \Z$ is an $\epsilon$-neighborhood of $\ell$ with the same property.
    \label{rmk:reflections-distance}
\end{remark}

Differently from the affine case, the roots corresponding to horizontal reflections are not necessarily orthogonal to the Coxeter plane.
In fact, $[1,w]$ contains infinitely many horizontal reflections.
However, we now show that $wr$ is a rotation if $r$ is vertical and a translation if $r$ is horizontal (as in the affine case).

\begin{lemma}\label{lem:reflections-vh}
    If $r$ is vertical, then $wr$ (resp.\ $rw$) is a rotation around an axial vertex $p$ (possibly at infinity).
    Specifically, $p$ is the vertex opposite to $\Fix(r)$ in the axial chamber immediately above (resp.\ below) $\Fix(r) \cap \ell$.
    If $r$ is horizontal, then $wr$ (resp.\ $rw$) is a translation whose axis meets the Coxeter axis $\ell$ with an angle $< \frac{\pi}{2}$.
    \label{lemma:reflections}
\end{lemma}

\begin{proof}
    If $r$ is horizontal, then the statement follows by \Cref{lemma:translations}.
    Suppose from now on that $r$ is vertical.
    By \Cref{lemma:reflections-and-axis}, the intersection point $\Fix(r) \cap \ell$ is fixed by at most one other reflection $r' \in W$ and, if this happens, the two reflection lines are perpendicular.
    Therefore, the axial chamber $C$ immediately above $\Fix(r) \cap \ell$ has $\Fix(r)$ as one of its walls.
    By \Cref{lemma:axial-factorization}, the walls of $C$ yield a factorization $w = r_1 r_2 r$.
    Therefore $wr = r_1r_2$ is a rotation around the vertex $\Fix(r_1) \cap \Fix(r_2)$ of $C$ opposite to $\Fix(r)$.
    Note that this vertex might be at infinity (if there is no relation between $r_1$ and $r_2$) and in such case, $wr$ is a parabolic isometry.
\end{proof}

\subsection{Rotations and translations}
Next, we give further insights on rank-two elements of the noncrossing partition poset $[1,w]$, namely, rotations and translations.

\begin{lemma}[Rotations]
    \label{lemma:rotations}
    Let $u \in [1, w]$ be a rotation around a point $p \in \H^2 \cup \partial \H^2$.
    Then $u$ is a Coxeter element for the parabolic subgroup $W_u \subseteq W$ that fixes $p$.
    In addition, a point $p \in \H^2 \cup \partial \H^2$ is the fixed point of a rotation $u \in [1, w]$ if and only if $p$ is an axial vertex.
\end{lemma}

\begin{proof}
    Let $r = w u^{-1}$ be the left complement of $u$.
    Then $r$ is a vertical reflection by \Cref{lemma:reflections} and $p$ is an axial vertex.
    Let $C$ be the axial chamber immediately below $\Fix(r)$.
    Note that $\Fix(r)$ is a wall of $C$ by \Cref{lemma:reflections-and-axis}.
    By \Cref{lemma:axial-factorization}, there is a factorization $w=rr_1r_2$ associated with the walls of $C$, and $r$ comes first.
    Therefore $u = r_1r_2$.
    We finish the proof of the first part of the statement by noting that $\{ r_1, r_2 \}$ is a set of simple reflections for the parabolic subgroup $W_u$ that fixes $p$.
    
    It remains to show that every axial vertex $p$ is the fixed point of a rotation $u \in [1, w]$.
    Let $C$ be an axial chamber having $p$ as one of its vertices.
  	By \Cref{lemma:axial-factorization}, there is a factorization $w = s_1s_2s_3$ where $s_1, s_2, s_3$ are the reflections with respect to the walls of $C$.
  	The three rotations $s_1s_2, s_2s_3, s_1s_3$ are all in $[1,w]$ and one of them is a rotation around $p$.
\end{proof}

The case of translations (given by the following lemma) is less trivial. Our proof relies on the fact that Coxeter elements minimize the spectral radius among all essential elements of $W$.

\begin{lemma}[Translations]
    \label{lemma:reflections-below-translation}
    Let $t \in [1, w]$ be a translation.
    A reflection $r \in W$ is in $[1, t]$ if and only if $\Fix(r)$ is orthogonal to $\Min(t)$.
    In addition, $t$ is a Coxeter element for the type-$\tilde A_1$ subgroup $W_t \subseteq W$ generated by the reflections below $t$.
\end{lemma}

\begin{proof}
    The product of two reflections $r, r'$ in the hyperbolic plane is a translation if and only if $\Fix(r)$ and $\Fix(r')$ do not meet (not even at infinity); when $t = rr'$ is a translation, the translation axis $\Min(t)$ is orthogonal to both $\Fix(r)$ and $\Fix(r')$.
    In particular, if $t \in [1,w]$ is a translation as in the statement, $\Fix(r)$ is orthogonal to $\Min(t)$ for any reflection $r \in [1, t]$. %
    Conversely, if $r\in W$ is a reflection such that $\Fix(r)$ is orthogonal to $\Min(t)$, then $r' = rt \in W$ is also a reflection and thus $r, r' \in [1, t]$.
    This proves the first part of the statement.

    The reflections below $t$ form an infinite discrete sequence $\dotsc, r_{-1}, r_0, r_1, r_2, \dotsc$ ordered according to the positions of $\Fix(r_i) \cap \Min(t)$ along the oriented translation axis $\Min(t)$.
    We need to prove that $t = r_{i+1}r_i$ for any (or equivalently all) $i \in \Z$.
    Suppose that $t = r_jr_i$ for some $j > i$.
    
    Let $r$ be the left complement of $t$, so that $w = rt$.
    Let $t_\lambda$ be the translation with the same translation axis as $t$ but with spectral radius changed to an arbitrary $\lambda > 1$ (i.e., the translation length is $\log \lambda$).
    Up to a change of coordinates in the half-plane model, we have
    \[
    t_\lambda = \begin{mymatrix}{cc}
    \lambda & 0 \\ 0 & 1
    \end{mymatrix}.
    \]
    Writing the reflection $r$ as in \eqref{eq:reflection}, we can express $rt_\lambda$ as 
    \begin{equation}
        rt_\lambda = \begin{mymatrix}{cc}
        \lambda\alpha & \rho^2 - \alpha^2 \\ \lambda & -\alpha
        \end{mymatrix}
        \label{eq:hyperbolic-isometry}        
    \end{equation}
    for some $\rho > |\alpha| > 0$.
    This is an isometry of the hyperbolic plane with an odd reflection length. It is not a reflection because $\Fix(r)$ is not orthogonal to the translation axis of $t_\lambda$, so it is a glide reflection.
    The spectral radius of $rt_\lambda$ is given by the absolute value of the ratio of the two eigenvalues of the matrix \eqref{eq:hyperbolic-isometry}.
    The characteristic polynomial of \eqref{eq:hyperbolic-isometry} is $t^2 - (\lambda-1)\alpha t - \lambda \rho^2$. The absolute value of the ratio between the larger and the smaller eigenvalue is
    \[
        \left( \frac{\sqrt{(\lambda-1)^2 + 4\lambda\beta} + (\lambda-1)}{2\sqrt{\lambda\beta}} \right)^2
    \]
    where $\beta = \rho^2 / \alpha^2 > 1$.
    This quantity is strictly increasing in $\lambda$ for $\lambda \geq 1$ (and it is equal to $1$ for the degenerate case $\lambda = 1$ where $t_\lambda$ becomes the identity). Thus the spectral radius of $rt_\lambda$ increases with the translation length of $t_\lambda$.
    
    Recall that $t=r_jr_i$ and suppose for the sake of contradiction that $j > i+1$. Then the translation length of $t' = r_{i+1}r_i$ is strictly smaller than the translation length of $t$.
    Therefore, the spectral radius of $w' = rt'$ is strictly smaller than the spectral radius of $w = rt$ (and strictly greater than $1$).
    As noted in the proof of \Cref{lemma:pedal-triangle}, all the Coxeter elements of $W$ have the same translation length and thus the same spectral radius.
    We reach a contradiction because $w'$ is essential and, by \cite[Theorem 4.1]{mcmullen2002coxeter}, the Coxeter elements minimize the spectral radius among all essential elements of $W$.
\end{proof}

The previous two lemmas yield the following analog of \cite[Theorem 3.22]{paolini2021proof}.

\begin{theorem}
    For every element $u \in [1,w]$, we have that $u$ is a Coxeter element for the subgroup of $W$ generated by the reflections below $u$.
    \label{thm:coxeter-elements-below-w}
\end{theorem}

We end this section by proving a geometric property of factorizations of translations into horizontal reflections.

\begin{lemma}[Five lines]
    Let $t \in [1, w]$ be a translation.
    Suppose that there is a factorization $t = r_1r_2$ where $r_1$ and $r_2$ are horizontal reflections whose fixed lines are on opposite sides of $\ell$.
    Let $r = wt^{-1}$ be the left complement of $t$ and $r' = t^{-1}w$ be the right complement of $t$.
    Then the (oriented) translation axis of $t$ intersects the five lines $\Fix(r'), \Fix(r_2), \ell, \Fix(r_1)$, and $\Fix(r)$ in this order, and these five lines are pairwise disjoint.
    \label{lemma:five-lines}
\end{lemma}

\begin{proof}
    The situation is depicted in \Cref{fig:five-lines}.
    By construction and by \Cref{lemma:reflections}, all reflections $r, r', r_1, r_2$ are horizontal and therefore their fixed lines do not intersect $\ell$.
    By \Cref{lemma:translations}, the translation axis of $t$ intersects $\Fix(r')$ and $\Fix(r_2)$ before $\ell$, and it intersects $\Fix(r)$ and $\Fix(r_1)$ after $\ell$.
    Since $w = rt = rr_1r_2$, we have that $rr_1$ is a translation by \Cref{lemma:translations}, and therefore $\Fix(r)$ does not intersect $\Fix(r_1)$. Similarly, $\Fix(r')$ does not intersect $\Fix(r_2)$.
    Applying \Cref{lemma:translations} to $w$ and $r_2$, we find that the (oriented) translation axis of $rr_1$ intersects $\Fix(r_2)$ before $\ell$.
    Therefore, $\Fix(r_1)$ is between $\ell$ and $\Fix(r)$.
    Similarly, $\Fix(r_2)$ is between $\ell$ and $\Fix(r')$.
\end{proof}

\begin{figure}
    \centering
    \begin{tikzpicture}[scale=3.6]
        \begin{scope}
            \clip (0,0) circle (1);
            
            \draw[
                white,
                decoration={markings, mark=at position 0.52 with {\arrow[orange, thick]{>}}},
                postaction={decorate}
            ] (-142.629284:1) -- (9.629284:1);

            \draw[
                white,
                decoration={markings, mark=at position 0.465 with {\arrow[blue, thick]{<}}},
                postaction={decorate}
            ] (-37.741750:1) -- (127.741750:1);

            \begin{scope}[every path/.style={thick, dashed, orange}]
                \hgline{-130.429284}{-4.570716}{}
            \end{scope}
            
            \input{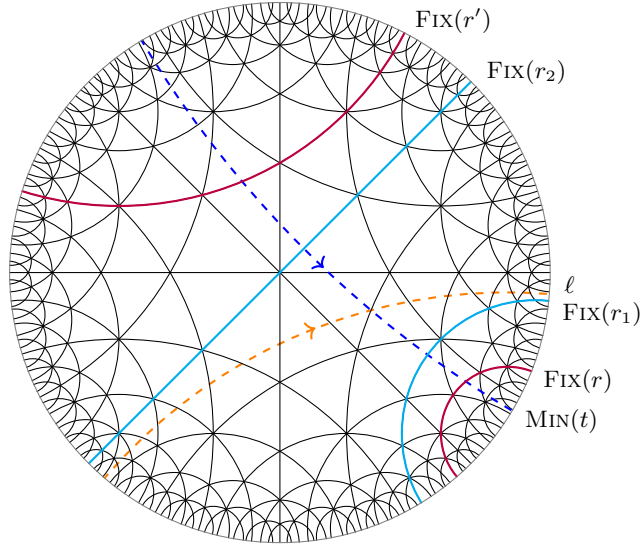}

            \begin{scope}[every path/.style={thick, dashed, blue}]
                \hgline{-30.741750}{120.741750}{}
            \end{scope}
        
            \begin{scope}[every path/.style={thick, purple}]
                \hgline{-49.100550}{-21.428229}{}
                \hgline{62.560359}{162.439641}{}
            \end{scope}
        
            \begin{scope}[every path/.style={thick, cyan}]
                \hgline{45.000000}{-135.000000}{}
                \hgline{-6.011412}{-58.459809}{}
            \end{scope}
            
            \draw[black!50, thick] (0,0) circle (1);
        \end{scope}
        
        \node at (0.64, 0.93) {\small $\Fix(r')$};
        
        \node at (0.91, 0.74) {\small $\Fix(r_2)$};
        \node at (1.07, -0.045) {\small $\ell$};
        \node at (1.19, -0.155) {\small $\Fix(r_1)$};
        
        \node at (1.1, -0.4) {\small $\Fix(r)$};
        \node at (1.04, -0.55) {\small $\Min(t)$};
    \end{tikzpicture}
    
    \caption{The oriented axis $\Min(t)$ of a translation $t = r_1r_2$ (dashed in blue) together with the other five lines of \Cref{lemma:five-lines}. The Coxeter axis $\ell$ is the same as in \Cref{fig:coxeter-axis}, left (dashed in orange).}
    \label{fig:five-lines}
\end{figure}

\section{The poset of noncrossing partitions}
\label{sec:poset}

In this section, we study some order-theoretic properties of the noncrossing partition poset $[1,w]$, where $w$ is a Coxeter element of a rank-three hyperbolic Coxeter group.
The following preliminary observation will allow us to link the order relation in $[1,w]$ with the geometric attributes of its elements.

\begin{lemma}
    The map $u \mapsto \Mov(u)$ from $[1, w]$ to the poset of linear subspaces of $\R^3$ (ordered by inclusion) is a poset isomorphism onto its image.
    \label{lemma:mov}
\end{lemma}

\begin{proof}
    Denote by $[1, w]_{\Isom(\H^2)}$ the interval between the identity and $w$ in the group $\Isom(\H^2)$ of all isometries of the hyperbolic plane $\H^2$ (using all reflections as generators).
    Since the reflection length of $w$ is $3$ both in $W$ and in $\Isom(\H^2)$, every geodesic between $1$ and $w$ in the Cayley graph of $W$ is also a geodesic in the Cayley graph of $\Isom(\H^2)$.
    Therefore, there is a natural order-preserving and rank-preserving inclusion $[1, w] \hookrightarrow [1, w]_{\Isom(\H^2)}$.
    To show that this inclusion is a poset isomorphism onto its image, we need to check that $u \leq v$ in $[1, w]_{\Isom(\H^2)}$ implies $u \leq v$ in $[1,w]$.
    The only non-trivial case is when $u$ has reflection length $1$ and $v$ has reflection length $2$ (in both $W$ and $\Isom(\H^2)$): in this case, $v = ur$ for some reflection $r \in \Isom(\H^2)$ and $r$ is also a reflection of $W$ (because $u, v \in W$), so $u \leq v$ in $[1,w]$.
    Finally, by \cite[Theorem 5.3]{mccammond2021factoring}, the map $u \mapsto \Mov(u)$ from $[1, w]_{\Isom(\H^2)}$ to the poset of linear subspaces of $\R^3$ is a poset isomorphism onto its image.
\end{proof}

\subsection{The lattice property}
We now prove that $[1,w]$ is a lattice and thus defines a Garside structure on the dual Artin group $W_w$.

\begin{theorem}
    The interval $[1,w]$ is a lattice.
    \label{thm:lattice}
\end{theorem}

\begin{proof}
    The poset $[1,w]$ is bounded and of rank $3$. Thus, if it is not a lattice then there are two reflections $r_1 \neq r_2$ and two rank-two elements $u_1 \neq u_2$ with $r_i < u_j$ for all $i,j\in \{1,2\}$ (this configuration is usually called a ``bowtie'').
    By \Cref{lemma:mov}, we have that $\Mov(r_i) \subseteq \Mov(u_j)$, where each $\Mov(r_i)$ is one-dimensional and each $\Mov(u_j)$ is two-dimensional.
    Also, $\Mov(r_1) \neq \Mov(r_2)$, and therefore $\Mov(u_1) = \Mov(r_1) + \Mov(r_2) = \Mov(u_2)$, which is a contradiction (again by \Cref{lemma:mov}).
\end{proof}

\begin{corollary}
    The dual Artin group $W_w$ is a Garside group.
    \label{cor:dual-garside}
\end{corollary}

\begin{proof}
    This follows immediately from \Cref{thm:garside,thm:lattice}, since $W_w$ is the interval group associated with the poset $[1,w]$, which is balanced (see \Cref{sec:coxeter-groups}).
\end{proof}

\subsection{Axial ordering}
\label{sec:axial-ordering}
Denote by $R_0$ the set of all reflections of $[1,w]$. We now describe a total ordering of $R_0$ which we call \emph{axial ordering}.

Working in the hyperboloid model, let $P \subseteq \R^3$ be the Coxeter plane, i.e., the linear span of the Coxeter axis $\ell$.
Given a reflection $r$, let $\overline\Fix(r) = \ker(\bar r - \id)$ where $\bar r$ is the extension of $r$ to a linear reflection of $\R^3$ (with respect to the quadratic form $Q$ of \Cref{iso-hyperboloid}).

Note that $\overline\Fix(r)$ is the fixed set of $\bar r$ in $\R^3$ and coincides with the linear span of $\Fix(r)$ in the hyperboloid model.

\begin{definition}[Axial ordering of reflections]
    Let $\prec$ be the total ordering of $R_0$ defined as follows.
    Fix an axial chamber $C_0$ and a point $q \in \ell \cap C_0$.
    Define $r_1 \prec r_2$ whenever the line $\overline\Fix(r_1) \cap P$ comes before the line $\overline\Fix(r_2) \cap P$ when traversing the projective line $\mathbb{P}(P)$ starting from the line spanned by $q$ and going in the positive direction of $\ell$.
    If $\overline\Fix(r_1) \cap P = \overline\Fix(r_2) \cap P$, define the relative order between $r_1$ and $r_2$ arbitrarily.
    \label{def:axial-ordering}
\end{definition}

\begin{remark}[Dihedral subgroups]
	Let $u \in [1,w]$ be a rotation, so that $[1,u]$ generates a dihedral subgroup $W_u$ for which $u$ is a Coxeter element (\Cref{lemma:rotations}).
	By \cite[Proposition 4.4]{paolini2021proof}, the restriction $\prec_u$ of the total order $\prec$ to $R_0 \cap [1, u]$ is a \emph{reflection ordering} for $W_u$: whenever $\alpha, \alpha_1, \alpha_2$ are distinct positive roots and $\alpha$ is a positive linear combination of $\alpha_1$ and $\alpha_2$, we have either
	\[ r_{\alpha_1} \prec r_\alpha \prec r_{\alpha_2} \quad \text{or} \quad r_{\alpha_2} \prec r_\alpha \prec r_{\alpha_1} \]
	(see \cite{bourbaki1968elements, dyer1993hecke, bjorner2006combinatorics} for more background on reflection orderings).
	The situation is depicted in \Cref{fig:rotation-and-axis}.
	Reflection orderings play an important role in the proof that finite noncrossing partition lattices are EL-shellable \cite{athanasiadis2007shellability}.
	\label{rmk:dihedral-reflection-orderings}
\end{remark}

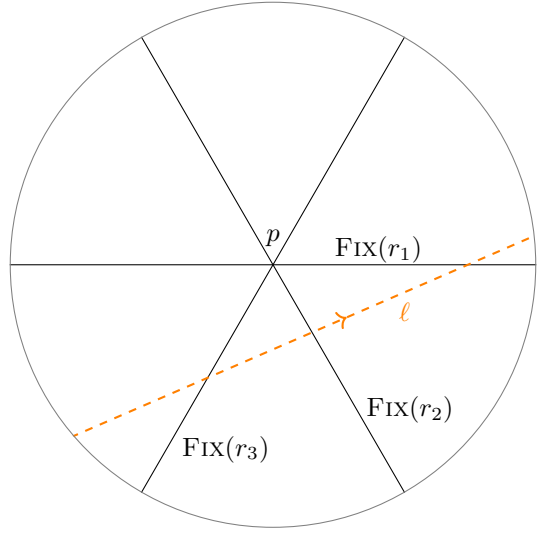
\begin{figure}
	\centering
	\begin{tikzpicture}[scale=3.5]
		\clip (0,0) circle (1);
		
		\node[label=$p$] (p) at (0, 0) {};
		\draw (-1.2, 0) -- (1.2, 0);
		\draw[rotate=60] (-1.2, 0) -- (1.2, 0);
		\draw[rotate=120] (-1.2, 0) -- (1.2, 0);
		
		\draw[orange, dashed, thick,
		decoration={markings, mark=at position 0.6 with {\arrow[orange, thick]{>}}},
		postaction={decorate}
		] (-1.1, -0.8) -- (1.2, 0.2);
		
		\node[orange] (l) at (0.5, -0.18) {$\ell$};
		
		\node at (0.4, 0.06) {$\Fix(r_1)$};
		\node at (0.52, -0.55) {$\Fix(r_2)$};
		\node at (-0.18, -0.7) {$\Fix(r_3)$};
		
		\draw[black!50, thick] (0,0) circle (1);
	\end{tikzpicture}
	\caption{
		Reflections below a $\frac{2\pi}{3}$-rotation $u = r_1r_2$ in the Klein model, as in \Cref{rmk:dihedral-reflection-orderings} and in Case 1 of the proof of \Cref{thm:shellability}.
		The Coxeter axis $\ell$ does not pass through $p$, so it defines a reflection ordering of $R_0 \cap [1, u]$, say $r_1 \prec r_3 \prec r_2$ (it could also be $r_2 \prec r_1 \prec r_3$ or $r_3 \prec r_2 \prec r_1$, depending on the location of the chamber $C_0$ along the axis $\ell$).
		We have that $u = r_1r_2 = r_2r_3 = r_3r_1$, so any of the three possible reflection orderings above is compatible with $u$.
	}
	\label{fig:rotation-and-axis}
\end{figure}

Next, we give a more intrinsic reformulation of the definition of axial ordering.

\begin{definition}
    Given a reflection $r \in R_0$, define $\xi(r) \in \ell$ to be the point of the Coxeter axis $\ell$ which is closest to $\Fix(r)$.
    Note that $\Fix(r)$ does not meet $\ell$ at infinity by \Cref{rmk:reflections-distance}, so $\xi(r)$ is not at infinity.
    An example of the construction of $\xi(r)$ in the Poincaré model is shown in \Cref{fig:labeling-reformulation}.
    \label{defxi}
\end{definition}

\begin{lemma}
    Let $r_1, r_2 \in R_0$ be horizontal reflections.
    We have that $\overline\Fix(r_1) \cap P = \overline\Fix(r_2) \cap P$ if and only if $\xi(r_1) = \xi(r_2)$.
    In addition, $r_1 \prec r_2$ if and only if $\xi(r_1)$ is above $\xi(r_2)$.
    \label{lemma:labeling-reformulation}
\end{lemma}

\begin{proof}    
    Working in the Poincaré model, assume without loss of generality that the Coxeter axis $\ell$ is a diameter of the disk, as in \Cref{fig:labeling-reformulation}. 
    Let $r$ be a horizontal reflection and let $l$ be its fixed line. 
    Then $l$ is represented by a circle that intersects $\partial \H^2$ perpendicularly in two points $q_1, q_2$ on the same side of $\ell$.
    Let $p$ be the point where the radical axis $\rho$ of $l$ and $\partial\H^2$ meets the extension of the Coxeter axis $\ell$ outside the disk (i.e., the intersection of the Coxeter plane $P$ with the plane in which the Poincar\'e disk is placed when compared to the hyperboloid model).
    Now, in order to identify the point $\xi(r)$, let $\gamma$ be the circle centered at $p$ and perpendicular to $l$; if $\rho$ and $\ell$ are parallel, take $\gamma$ to be the axis of the segment connecting $q_1$ and $q_2$.
    Since the power of $p$ with respect to $l$ is the same as with respect to $\partial\H^2$, the circle $\gamma$ is perpendicular to $\partial\H^2$ as well (and of course to $\ell$). Thus, conformality of the Poincar\'e model implies that $\gamma\cap\ell=\xi(r)$.

    We now pass to the Klein model, noting that the canonical transformation from the Poincaré to the Klein disk maps the boundary points (such as $q_1$, $q_2$) identically and preserves the axis $\ell$.
    However, now $\rho$ is the transform of the line $l$, and so $p$ represents $\overline\Fix(r)\cap P$.
    Thus, the position of $p$ along $\ell$ determines the position of $r$ in the axial ordering of \Cref{def:axial-ordering}.
    Since $\xi(r) = \gamma\cap\ell$ only depends on the power of $p$ with respect to the boundary of the model disk, the first claim follows immediately.

    For the second claim, observe that if we move $p$  further away from the center of the disk, its power with respect to the boundary circle increases. Accordingly, the point $\xi(r)$ will move towards the center of the disk, i.e., in the opposite direction with respect to $p$.
\end{proof}

\begin{figure}
    \centering
    \begin{tikzpicture}[scale=3]
        \clip (-1.1, -1.05) rectangle (2, 1.12);
        \draw[gray, thick] (0, 0) circle (1);

        \draw[gray, thin] (0.376, 0.37) -- (0.425, 0.355) -- (0.442, 0.415);
        
        \draw[orange, dashed, thick,
              decoration={markings, mark=at position 0.3 with {\arrow[orange, thick]{>}}},
                postaction={decorate}
            ] (-1, 0) -- (2, 0);
        
        \draw (-0.7, 1.373) -- (2, -0.166);
        
        \node[inner sep=1, fill=black, circle] (p) at (1.7087, 0) {};
        \node[inner sep=1, fill=black, circle] at (0.32368, 0) {};
        
        \node[inner sep=1, fill=black, circle] (q1) at (-0.0456, 0.9996) {};
        \node[inner sep=1, fill=black, circle] (q2) at (0.8833, 0.4704) {};

        \node (c) at (0.5860, 1.0285) {};
        \draw (c) circle (0.632305);
        \draw (p) circle (1.38502);

        \node at (-0.6, 0.08) {$\ell$};
        \node at (-0.12, 0.5) {$l = \Fix(r)$};
        \node at (1.33, 0.3) {$\rho$};
        \node at (-0.7, 0.9) {$\partial\H^2$};
        \node at (1.75, 0.08) {$p$};
        \node at (0.04, 1.07) {$q_1$};
        \node at (0.91, 0.57) {$q_2$};
        \node at (0.45, 0.08) {$\xi(r)$};
        \node at (0.3, -0.5) {$\gamma$};

    \end{tikzpicture}
    \caption{Proof of \Cref{lemma:labeling-reformulation}.}
    \label{fig:labeling-reformulation}
\end{figure}

\begin{remark}
    \label{rem:secondorder}
    The previous lemma gives a reformulation of the ordering $\prec$ of \Cref{def:axial-ordering}: first come the vertical reflections $r$ such that $\Fix(r) \cap \ell$ is above $C_0$, ordered by $\Fix(r) \cap \ell$ using the orientation of $\ell$; then come the horizontal reflections $r$, ordered by $\xi(r)$ using the \textit{reverse} orientation of $\ell$; finally come the vertical reflections such that $\Fix(r) \cap \ell$ is below $C_0$, ordered by $\Fix(r) \cap \ell$ using the orientation of $\ell$.
\end{remark}

\begin{definition}
    Once a Coxeter element $w$ has been fixed, denote by $\varphi (u) := w^{-1}uw$ the conjugation by $w$.
    \label{def:phi}
\end{definition}

The following lemmas describe how $\varphi$ interacts with the axial ordering $\prec$.

\begin{lemma}
    Let $r_1, r_2 \in R_0$ be two horizontal reflections.
    We have $r_1 \prec r_2$ if and only if $\varphi(r_1) \prec \varphi(r_2)$.
    \label{lemma:phi-preserves-order}
\end{lemma}

\begin{proof}
    Conjugating a horizontal reflection $r$ by $w$ has the effect of translating $\xi(r)$ in the negative direction of $\ell$ by an amount equal to the translation length of $w$: $\xi(\varphi(r)) = w^{-1}(\xi(r))$.
    The statement then follows from \Cref{rem:secondorder}.
\end{proof}

\begin{lemma}
    If $r \in R_0$ is a vertical reflection, then $\varphi(r) \prec r$ unless $r$ is among the three $\prec$-smallest reflections.
    If $r$ is horizontal, then $\varphi(r) \succ r$.
    \label{lemma:phi-and-order}
\end{lemma}

\begin{proof}
    If $r$ is vertical, then $\varphi$ moves $\Fix(r) \cap \ell$ in the negative direction of $\ell$ by the translation length of $w$.
    Therefore, $\varphi(r) \succ r$ if and only if $\Fix(r) \cap \ell$ is between $C_0$ and $w(C_0)$.
    By \Cref{lemma:pedal-triangle}, there are exactly three reflection hyperplanes separating $C_0$ and $w(C_0)$, so the result follows.
    If $r$ is horizontal, we immediately have $\varphi(r) \succ r$ by \Cref{rem:secondorder}.
\end{proof}

\begin{lemma}
    For every vertical reflection $r \in R_0$, there exists a unique $j \in \Z$ such that $\varphi^j(r)$ is one of the three $\prec$-smallest reflections. In addition, $j \geq 0$ if and only if $\Fix(r)$ intersects $\ell$ above $C_0$.
    \label{lemma:conjugate-smallest-reflections}
\end{lemma}

\begin{proof}
    As already noted in the proof of the previous lemma, exactly three reflection hyperplanes separate $C_0$ and $w(C_0)$, and they are the fixed lines of the three $\prec$-smallest reflections.
    The statement immediately follows.
\end{proof}

\begin{lemma}
    Let $r \in R_0$ be a vertical reflection. Its right complement $u = rw$ fixes a vertex of $C_0$ if and only if $r$ is among the three $\prec$-smallest reflections.
    \label{lemma:smallest-reflections}
\end{lemma}

\begin{proof}
    Let $C$ be the axial chamber immediately below $\Fix(r) \cap \ell$.
    By \Cref{lemma:reflections}, $u$ is a rotation around the vertex $p$ of $C$ opposite to the wall $\Fix(r)$.
    
    Suppose that $p$ is a vertex of $C_0$.
    In particular, $\Fix(r)$ intersects $\ell$ above $C_0$, because otherwise $p$ and $C_0$ would be on opposite sides of $\Fix(r)$.
    By \Cref{lemma:axial-vertices-orbits}, $C_0$ and $w(C_0)$ have no vertices in common, so $C$ is between $C_0$ (included) and $w(C_0)$ (excluded).
    Therefore, $\Fix(r)$ intersects $\ell$ below $w(C_0)$.
    We conclude that $r$ is one of the three $\prec$-smallest reflections.
    
    Conversely, suppose that $r$ is one of the three $\prec$-smallest reflections. Then $\Fix(r)$ intersects $\ell$ between $C_0$ and $w(C_0)$, so $C$ is between $C_0$ (included) and $w(C_0)$ (excluded).
    Therefore, $p$ and $w(C_0)$ lie on opposite sides of $\Fix(r)$, so $p$ is not a vertex of $w(C_0)$. By \Cref{lemma:axial-vertices-orbits}, $p$ is a vertex of $C_0$.
\end{proof}

\subsection{EL-labeling}
Next, we prove that the interval $[1,w]$ admits an EL-labeling (see \Cref{sec:EL}), so in particular, it is EL-shellable.
Note that noncrossing partition posets are known to be EL-shellable in the spherical case \cite{athanasiadis2007shellability} and in the affine case \cite{paolini2021proof}.

\begin{definition}
    The natural labeling of $[1,w]$ by reflections is the labeling $E([1,w]) \to R_0$ which maps a poset cover $u \lessdot ur$ to the reflection $r \in R_0$.
    The set of reflections $R_0$ is totally ordered by the axial ordering $\prec$ constructed in \Cref{sec:axial-ordering}.
\end{definition}

\begin{theorem}
    The natural labeling of $[1, w]$ by reflections ordered by $\prec$ is an EL-labeling.
    In other words, every element $u \in [1, w]$ has a unique $\prec$-increasing factorization into reflections and this factorization is lexicographically minimal.
    \label{thm:shellability}
\end{theorem}

\begin{proof}
\emph{Case 1:} $u$ is a rotation.
Then $u$ is a rotation around a point $p \in \H^2 \cup \partial \H^2$ and it is a Coxeter element for the (parabolic) dihedral subgroup $W_u$ that fixes $p$ by \Cref{lemma:rotations}.
If $u$ is a rotation through an angle of $\pi$, then there are exactly two factorizations $u = r_1r_2 = r_2r_1$; exactly one of them is $\prec$-increasing and lexicographically smaller than the other one.

Suppose from now on that $u$ is a rotation through an angle less than $\pi$.
The restriction of $\prec$ to $R_0 \cap [1, u]$ is a reflection ordering (see \Cref{rmk:dihedral-reflection-orderings} and \Cref{fig:rotation-and-axis}).
By \cite[Theorem 3.5]{athanasiadis2007shellability}, it is enough to show that this ordering is compatible with $u$, i.e., that there is at least one factorization $u = r_1r_2$ such that $r_1$ comes immediately after $r_2$ in the cyclic order $\prec'$ where the last reflection comes immediately before the first (see the caption of \Cref{fig:rotation-and-axis}).
This is true if $r_1$ and $r_2$ are as in the proof of \Cref{lemma:rotations}, because $\Fix(r_2)$ intersects $\ell$ immediately below the axial chamber $C$.

\emph{Case 2:} $u$ is a translation.
By \Cref{lemma:reflections-below-translation}, $u$ is a product of any two consecutive reflections below $u$.
Recall from \Cref{lemma:translations} that the translation axis of $u$ and the Coxeter axis $\ell$ intersect with an angle less than $\frac{\pi}{2}$.
Then, if $u = r_1r_2$, the reflection $r_1$ comes immediately after $r_2$ in the cyclic order $\prec'$ on $R_0 \cap [1, u]$ (this is easily seen in the Klein model, see \Cref{fig:shellability-translations}).
Therefore $r_1 \succ r_2$, unless $r_1$ is the $\prec$-first reflection and $r_2$ is the last.

\begin{figure}
    \centering
    \begin{tikzpicture}[scale=3.5]
    	\clip (0,0) circle (1);
    	
    	\draw[orange, dashed, thick,
    	decoration={markings, mark=at position 0.71 with {\arrow[orange, thick]{>}}},
    	postaction={decorate}
    	] (-1.2, -0.5) -- (1.2, 0.5);

    	\draw[blue, dashed, thick,
		decoration={markings, mark=at position 0.775 with {\arrow[blue, thick]{>}}},
		postaction={decorate}
		] (-1, 0) -- (1, 0);
    	
    	\node[orange] (l) at (0.5, 0.3) {$\ell$};
    	\node[blue] at (0.7, -0.1) {$\Min(u)$};
    	
    	\draw (-0.7, 1) -- (-0.7, -1);
    	\node at (-0.5, 0.4) {$\Fix(r_2)$};
    	
    	\draw (-0.25, 1) -- (-0.25, -1);
		\node at (-0.05, 0.4) {$\Fix(r_1)$};

    	\draw[black!50, thick] (0,0) circle (1);
    \end{tikzpicture}
    \caption{
  	Case 2 of the proof of \Cref{thm:shellability}, depicted in the Klein model.
  	Both the Coxeter axis $\ell$ and the translation axis $\Min(u)$ are taken as diameters of the disk.
    Then any reflection $r \leq u$ fixes a line perpendicular to $\Min(u)$.
    In addition, the oriented axes $\ell$ and $\Min(u)$ form an angle smaller than $\frac{\pi}{2}$.
	Therefore, if $u = r_1r_2$, then $\Fix(r_1)$ intersects (the extension of) $\ell$ above $\Fix(r_2)$.
	}
    \label{fig:shellability-translations}
\end{figure}
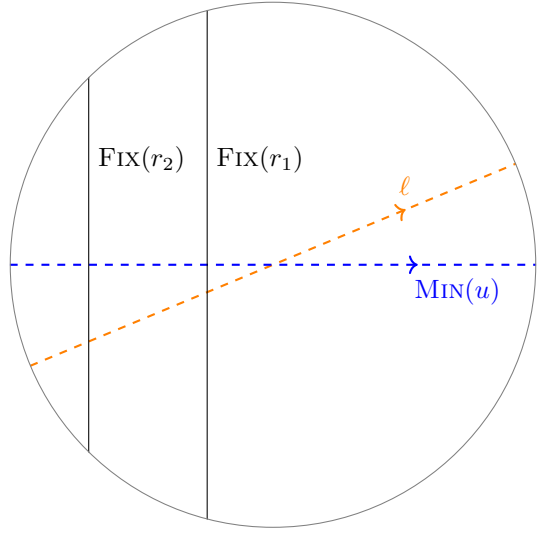

\emph{Case 3:} $u = w$.
The lexicographically minimal factorization $w = abc$ is increasing by construction (if $a$ and $b$ commute, or $b$ and $c$ commute, we may need to swap $a$ and $b$, or $b$ and $c$, for this to become true).
Therefore, it is enough to show that every increasing factorization $w=r_1r_2r_3$ coincides with the factorization $w = abc$.

Suppose that $r_1$ is vertical and let $C$ be the axial chamber immediately below $\Fix(r_1)$.
By \Cref{lemma:reflections}, $v = r_2r_3$ is a rotation around the vertex $p$ of $C$ opposite to $\Fix(r_1)$.
In addition, we can write $v = r_2' r_3'$ where $\Fix(r_2')$ and $\Fix(r_3')$ are the two walls of $C$ containing $p$.
By \Cref{lemma:axial-factorization}, $\Fix(r_3')$ intersects the Coxeter axis $\ell$ below $\Fix(r_1)$, so $r_3' \prec r_1$ unless $C = C_0$.
By Case 1, since $r_2 \prec r_3$, we have that $r_2$ is the $\prec$-smallest reflection below $v$; in particular, $r_2 \preceq r_3'$.
Putting this all together, if $C \neq C_0$, we get $r_2 \preceq r_3' \prec r_1$ which is a contradiction.
If $C = C_0$, the factorization $r_1r_2r_3$ coincides with $abc$.
A similar argument can be carried out if $r_3$ is vertical.

Suppose now that $r_1$ and $r_3$ are both horizontal.
Since $r_1 \prec r_2 \prec r_3$, we have that $r_2$ is also horizontal.
By Case 2, if $t = rr'$ is the increasing factorization of a translation and both $r$ and $r'$ are horizontal, then $\Fix(r)$ and $\Fix(r')$ are on opposite sides of the Coxeter axis.
Applying this observation to the three factorizations $r_1r_2$ (translation), $r_2r_3$ (translation), and $r_1r_3$ (translation or rotation), we find that $r_1r_3$ has to be a rotation.
Then its left complement $r_1r_2r_1$ is a vertical reflection.
However, $\Fix(r_1)$ and $\Fix(r_2)$ are on opposite sides of the Coxeter axis $\ell$, so $\Fix(r_1r_2r_1) = r_1(\Fix(r_2))$ and $\ell$ are on opposite sides of $\Fix(r_1)$, and in particular they cannot intersect.
This is a contradiction.
\end{proof}

\section{The interval complex $K_W$ and its subcomplexes}
\label{sec:interval-complex}

Throughout this section, let $(W, S)$ be a non-spherical irreducible Coxeter system of arbitrary rank $n$, where $S = \{s_1, s_2, \dotsc, s_n\}$ is the set of reflections with respect to some chamber $C_0$.
Fix a Coxeter element $w = s_1s_2\dotsm s_n$.
As usual, $[1,w]$ denotes the interval between $1$ and $w$ in the right Cayley graph of $W$ with respect to the set of all reflections $R \subseteq W$.
Let $K_W$ be the interval complex associated with $[1,w]$ (see \Cref{sec:garside-structures}).

For every $T \subseteq S$, consider the Coxeter element $w_T$ of the standard parabolic subgroup $W_T$ consisting of the product of the elements of $T$ in the same relative order as in the sequence $s_1, s_2, \dotsc, s_n$.
As shown in \cite[Section 5]{paolini2021proof}, the interval $[1,w_T]$ is the same in $W_T$ and in $W$ (using all reflections as generators) and the interval complex $K_{W_T}$ is naturally a subcomplex of $K_W$.
Let $X_W' \subseteq K_W$ be the subcomplex defined as
\[
    X_W' = \bigcup_{\substack{\phantom{\frac{a}{b}}{T \subseteq S}\phantom{\frac{a}{b}} \\ \text{$W_T$ finite}}} K_{W_T}.
\]

\begin{theorem}[{\cite[Theorem 5.5]{paolini2021proof}}]
    For any Coxeter system $(W, S)$ and Coxeter element $w$, the subcomplex $X_W' \subseteq K_W$ is homotopy equivalent to the orbit configuration space $Y_W$.
    \label{thm:XW'}
\end{theorem}

A deformation retraction $K_W \searrow X_W'$ was a crucial step in the proof of the $K(\pi, 1)$ conjecture in the affine case \cite[Theorem 8.14]{paolini2021proof} and it is natural to ask whether such a deformation retraction can be constructed for a general $W$ (see \cite[Question 5.5]{paolini2021dual}).
In this section, we introduce three additional subcomplexes of $K_W$ that will be useful in answering the previous question when $W$ has rank three (\Cref{sec:dmt-rank-three}).
The overall structure of the retraction argument is summarized in \Cref{treasuremap}. The reader might find it helpful to refer to \Cref{treasuremap} when reading the definitions of the three new complexes below.
We define the subcomplexes in full generality, with the hope that they will also prove useful when dealing with arbitrary Coxeter groups.

\begin{figure}
    \begin{center}
    \def\sz{\footnotesize}
\begin{tikzpicture}
\node (b1) at (-2,.8) {\sz{Sec.\ \ref{sec:matching}-\ref{sec:acyclicity}}};
\node (d1) at (0,.8) {\sz{Prop.\ \ref{prop:K''-K'}} ${}^*$};
\node (f1) at (2,.8) {\sz{Sec.\ \ref{sec:matching-N}}};
\node (a) at (-3,-.5) {$K_W$};
\node (b) at (-2,-.5) {$\searrow$};
\node (c) at (-1,-.5) {$\Klong$};
\node (d) at (0,-.5) {$\searrow$};
\node (e) at (1,-.5) {$\Kshort$};
\node (f) at (2,-.5) {$\searrow$};
\node (g) at (3,-.5) {$\Ksparse$};
\node[anchor=center] (t) at (0,-.7) {\phantom{For every infinite parabolic subgroup of $W$.}};
\draw[-,dotted] (b1.south) -- (b.north);
\draw[-,dotted]%
(d1.south) -- (d.north);
\draw[-,dotted]%
 (f1.south) -- (f.north);
\node (C) at (-.8,-3) {$K_W$};
\node (D) at (0,-3) {$\searrow$};
\node (E) at (0.8,-3) {$X_W'$};
\node (F) at (1.5,-3) {$\simeq$};
\node (G) at (2.2,-3) {$Y_W$};
\draw [decorate,
    decoration = {calligraphic brace}] (t.south east) --  (t.south west);
\draw[->,dotted]%
    (0,-1.2) -- (0,-2);
\draw[-,dotted]%
    (0,-2) -- (D.north);
    \node[anchor=west] (p) at (0,-1.9) {\sz{Prop.\ \ref{prop:inductive-collapse}} ${}^*$};
\node (k) at (-2.6,-3) {$K(W_w,1)$};
\node (l) at (-1.5,-3) {$\simeq$};
\node (l1) at (-1.5,-4) {\sz{Thm.\ \ref{thm:garside} (Garside)}};
\node (F1) at (1.5,-4) {\sz{Thm.\ \ref{thm:XW'}} ${}^*$};
\draw[dotted] (l1.north) -- (l.south);
\draw[dotted] (F1.north) -- (F.south);
\end{tikzpicture}
    \end{center}
    \caption{
    	Relationships between the various subcomplexes of $K_W$ and strategy of the proof of \Cref{thm:K-X'}.
   		The statements marked with an asterisk are valid for a general Coxeter group $W$.
   		Note that \Cref{prop:inductive-collapse} requires the collapses on the top row of the picture to hold for every infinite standard parabolic subgroup of $W$. 
	}
    \label{treasuremap}
\end{figure}

\subsection{The subcomplex $\Ksparse$ and inductive collapses}\label{X2}
Consider the following subcomplex of $K_W$:
\[
    \Ksparse = \bigcup_{T \subsetneq S} K_{W_T}.
\]
Since $W$ is infinite, every finite standard parabolic subgroup $W_T$ is a proper subgroup of $W$ and therefore $X_W' \subseteq \Ksparse$.
Note that $X_W'=\Ksparse$ whenever $W$ is affine or compact hyperbolic because every proper parabolic subgroup of $W$ is finite.
Also, note that a simplex $[x_1|\dotsb|x_d]$ of $K_W$ belongs to $\Ksparse$ if and only if the product $x_1\dotsm x_d$ fixes a vertex of $C_0$.
The nickname ``sparse'' is motivated in \Cref{fig:subcomplexes} below.

The reason for introducing $\Ksparse$ is given by the following criterion, which allows us to deformation retract $K_W$ onto $X_W'$ inductively.

\begin{proposition}\label{prop:collassoT}
    Suppose that $K_{W_T} \searrow \Ksparse[W_T]$ for every infinite standard parabolic subgroup $W_T$ of $W$. Then $K_W \searrow X_W'$.
    \label{prop:inductive-collapse}
\end{proposition}

\begin{proof}
    Let $T_1, T_2, \dotsc, T_k$ be the list of all subsets of $S$ such that $W_T$ is infinite.
    Order this list in such a way that $T_i \subseteq T_j$ implies $i > j$.
    In particular, $T_1 = S$.
    Note that any deformation retraction of $K_{W_T}$ onto $\Ksparse[W_T]$ must delete all the simplices of $K_{W_T}$ that are not contained in any $K_{W_{T'}}$ with $T' \subsetneq T$ while fixing all $K_{W_{T'}}$ with $T'\subsetneq T$.
    Therefore, it is enough to start from $K_W$ and perform the deformation retractions $K_{W_{T}} \searrow \Ksparse[W_T]$ one at a time for $T=T_1, \dotsc, T_k$.
    The final subcomplex is precisely $X_W'$.
\end{proof}

Hence, in order to deformation retract $K_W$ onto $X_W'$ it is enough to be able to deformation retract $K_W$ onto the larger subcomplex $\Ksparse$ (and to be able to do so for all infinite $W_T \subseteq W$).

\subsection{Fiber components and the subcomplexes $\Kshort$ and $\Klong$}\label{defeta}
As in the affine case \cite[Section 7]{paolini2021proof}, consider the poset map $\eta \colon \mathcal{F}(K_W) \to \N$ defined as
\begin{equation}    
    \eta([x_1|x_2|\dotsb|x_d]) =
    \begin{cases}
    d & \text{if $x_1 x_2 \dotsm x_d = w$} \\
    d+1 & \text{otherwise}.
    \end{cases}
    \label{eq:eta}
\end{equation}
The connected components of any fiber $\eta^{-1}(d)$ in the Hasse diagram of $\mathcal{F}(K_W)$ are called the \emph{$d$-fiber components}.
A $d$-fiber component has the form
\begin{center}
	\newcommand{\componentwidth}{1.35}
	\newcommand{\height}{-1.5}
	\begin{tikzpicture}
		\clip (-4*\componentwidth, \height-0.3) rectangle (4*\componentwidth, 0.3);  %
		
		\node (0) at (0,0) {$[x_1|\dotsb|x_d]$};
		\node (1) at (\componentwidth, \height) {$[x_2|\dotsb|x_d]$};
		\node (2) at (2*\componentwidth, 0) {$[x_2|\dotsb|x_{d+1}]$};
		\node (3) at (3*\componentwidth, \height) {$[x_3|\dotsb|x_{d+1}]$};
		\node (4) at (4*\componentwidth,0) {\phantom{$[\,]$}};
		\coordinate (3a) at ($(3)!0.5!(4)$) {};

		\node (-1) at (-\componentwidth, \height) {$[x_1|\dotsb|x_{d-1}]$};
		\node (-2) at (-2*\componentwidth, 0) {$[x_0|\dotsb|x_{d-1}]$};
		\node (-3) at (-3*\componentwidth, \height) {$[x_0|\dotsb|x_{d-2}]$};
		\node (-4) at (-4*\componentwidth, 0) {\phantom{$[\,]$}};
		\coordinate (-3a) at ($(-3)!0.5!(-4)$) {};
		
		\draw (0.south) -> (1.north);
		\draw (2.south) -> (1.north);
		\draw (2.south) -> (3.north);
		\draw[dashed] (3.north) -> (3a);
		
		\draw (0.south) -> (-1.north);
		\draw (-2.south) -> (-1.north);
		\draw (-2.south) -> (-3.north);
		\draw[dashed] (-3.north) -> (-3a);
	\end{tikzpicture}
\end{center}
where $x_i x_{i+1}\dotsm x_{i+d-1} = w$ for all $i$.
The bi-infinite sequence $(x_i)_{i \in \Z}$ satisfies $x_{i+d} = \varphi(x_i)$ for all $i$, where $\varphi$ is the conjugation by $w$ introduced in \Cref{def:phi}.
The collection of all $d$-fiber components for $d=1, \dotsc, n$ yields a partition of $\mathcal{F}(K_W)$.
Note that a fiber component can be finite (so its Hasse diagram is a closed loop) or infinite.

We now use the notion of fiber components to introduce two further subcomplexes of $K_W$ that interpolate between $K_W$ and its subcomplex $\Ksparse$ defined in \Cref{X2}.

\begin{definition}[Subcomplexes $\Kshort$ and $\Klong$]\label{2K} \strut
\begin{enumerate}[(a), leftmargin=0.65cm]
	\item Define $\Kshort$ as the union of all simplices $\sigma$ of $K_W$ such that there is a simplex of $\Ksparse$ both (weakly) to the left and (weakly) to the right of $\sigma$ in the fiber component $\CC$ containing $\sigma$.
    Roughly speaking, $\Kshort$ is constructed by taking the ``convex hull'' of the simplices of $\Ksparse$ in the Hasse diagram of each fiber component.
	\item Define $\Klong$ as the 
    union of all fiber components that intersect $\Ksparse$.
\end{enumerate}
\end{definition}

The nicknames ``short'' and ``long'' are motivated in \Cref{fig:subcomplexes}.
The following lemma shows that $\Kshort$ and $\Klong$ indeed are subcomplexes of $K_W$.

\begin{lemma}
    Let $\sigma \in \F(K_W)$ and let $\tau$ be a face of $\sigma$.
    Suppose that the fiber component of $\sigma$ contains a simplex of $\Ksparse$ weakly to the right (respectively, to the left) of $\sigma$.
    Then the fiber component of $\tau$ contains a simplex of $\Ksparse$ weakly to the right (respectively, to the left) of $\tau$.
\end{lemma}

\begin{proof}
	Let $d = \eta(\sigma)$, so that $\sigma$ is of the form $[x_1|\dotsb|x_{d}]$ or $[x_1|\dotsb|x_{d-1}]$ where $x_1\dotsm x_d = w$.
    Denote by $\CC$ the $d$-fiber component of $\sigma$.
    Let $\sigma'$ be a simplex of $\Ksparse$ weakly to the right of $\sigma$ in $\CC$. Then $\sigma'$ is a minimal element of $\CC$ and takes the form
	\[
	\sigma'=[\varphi^{-k}(x_{j+1})\vert \cdots\vert \varphi^{-k}(x_d)\vert\varphi^{-k+1}(x_1)\vert \cdots\vert \varphi^{-k+1}(x_{j-1})]
	\]
	for some $j=1,\ldots,d$ and some $k>0$.
    
	Recall from \Cref{sec:garside-structures} that the faces of an $m$-simplex $[y_1|\dotsb|y_m]$ are denoted by $\partial_i([y_1|\dotsb|y_m])$ for $i=0,\dotsc, m$.
	Let $\tau=\partial_i(\sigma)$.
    If $\sigma$ has dimension $d$ and $i=0$ or $i=d$, then $\tau\in \CC$ and $\sigma'$ is weakly to the right of $\tau$.
    In all other cases, the fiber component of $\tau$ contains the simplex $\partial_{k}(\sigma')$ weakly to the right of $\tau$, where $k = i-j \text{ mod } d$.
\end{proof}

Recall from \cite[Section 7]{paolini2021proof} that, given an infinite $d$-fiber component, for any connected subgraph of its Hasse diagram starting and ending with a $(d-1)$-simplex, there exists a proper acyclic matching with critical simplices given by this subgraph (see \Cref{fig:component-matching}).
This immediately gives us the following result.

\begin{figure}
	\begin{center}
		\newcommand{\componentwidth}{1.2}
		\newcommand{\height}{-1.5}
		\begin{tikzpicture}
			\clip (-4*\componentwidth, \height-0.3) rectangle (6*\componentwidth, 0.3);  %
			
			\node (0) at (0,0) {$[x_1|\dotsb|x_d]$};
			\node (1) at (\componentwidth, \height) {$[x_2|\dotsb|x_d]$};
			\node (2) at (2*\componentwidth, 0) {$[x_2|\dotsb|x_{d+1}]$};
			\node (3) at (3*\componentwidth, \height) {$[x_3|\dotsb|x_{d+1}]$};
			\node (4) at (4*\componentwidth, 0) {$[x_3|\dotsb|x_{d+2}]$};
			\node (5) at (5*\componentwidth, \height) {$[x_4|\dotsb|x_{d+2}]$};
			\node (6) at (6*\componentwidth, 0) {\phantom{$[\,]$}};
			
			\coordinate (5a) at ($(5)!0.5!(6)$) {};

			\node (-1) at (-\componentwidth, \height) {$[x_1|\dotsb|x_{d-1}]$};
			\node (-2) at (-2*\componentwidth, 0) {$[x_0|\dotsb|x_{d-1}]$};
			\node (-3) at (-3*\componentwidth, \height) {$[x_0|\dotsb|x_{d-2}]$};
			\node (-4) at (-4*\componentwidth, 0) {\phantom{$[\,]$}};
			\coordinate (-3a) at ($(-3)!0.5!(-4)$) {};
			
			\draw (0.south) -> (1.north);
			\draw (2.south) -> (1.north);
			\draw (4.south) -> (3.north);
			\draw[dashed] (5.north) -> (5a);
			
			\draw (0.south) -> (-1.north);
			\draw (-2.south) -> (-1.north);
			\draw[dashed] (-3.north) -> (-3a);
			
			\begin{scope}[red, very thick]
				\draw[<-] (-2.south) -- (-3.north);
				\draw[<-] (2.south) -- (3.north);
				\draw[<-] (4.south) -- (5.north);
			\end{scope}
			
			\begin{scope}[circle]
				\node[fill=white, draw, inner sep=1.8] at (-1.north) {};
				\node[fill=white, draw, inner sep=1.8] at (1.north) {};
				\node[fill=white, draw, inner sep=1.8] at (0.south) {};
			\end{scope}
			
		\end{tikzpicture}
	\end{center}
	
	\caption{A matching on an infinite $d$-fiber component leaving the critical simplices $[x_1|\dotsb|x_{d-1}]$, $[x_1|\dotsb|x_{d}]$, and $[x_2|\dotsb|x_{d}]$ unmatched.
	More generally, one can construct a similar matching that leaves an arbitrary finite set of contiguous critical simplices unmatched, starting and ending with $(d-1)$-dimensional simplices.
	This matching is acyclic and proper.}
	\label{fig:component-matching}
\end{figure}

\begin{proposition}
	The complex $\Klong$ deformation retracts onto $\Kshort$.
	\label{prop:K''-K'}
\end{proposition}

\begin{proof}
	Every $d$-fiber component $\CC$ of $\Klong$ intersects $\Ksparse$ (and thus $\Kshort$), so it admits a proper acyclic matching with $\CC \cap \F(\Kshort)$ as the set of critical simplices (if $\CC$ is finite, then $\CC\subseteq \F(\Kshort)$, so the empty matching works).
	Let $\M$ be the union of all these matchings.
	By the Patchwork theorem (\Cref{thm:patchwork}), $\M$ is acyclic.
    Let $H_\M$ be the acyclic graph defined in \Cref{sec:dmt}, where $H$ is the Hasse diagram of $\F(\Klong)$.
    If $\sigma$ is any simplex in a $d$-fiber component, then every oriented path starting from $\sigma$ in the graph $H_\M$ can change the fiber component at most $d-1$ times (each time the fiber component changes, the value of $\eta$ decreases by $1$; the initial value is $d$ and the value on any simplex is $\geq 1$).
	By induction on $d$, for any $\sigma$ in a $d$-fiber component, there are only a finite number of simplices reachable from $\sigma$.
	Therefore, $\M$ is proper.
	By \Cref{thm:dmt}, we conclude that $\Klong \searrow \Kshort$.
\end{proof}

\begin{figure}
	\begin{center}
		\newcommand{\componentwidth}{0.67}
		\newcommand{\height}{-1.2}
		\newcommand{\first}{-3}
		\newcommand{\second}{-2}
		\newcommand{\last}{5}
		\begin{tikzpicture}
			\begin{scope}[shift={(0, 2.5)}]
				\clip (2*\first*\componentwidth, \height-0.3) rectangle (2*\last*\componentwidth+2*\componentwidth, 0.3);
				
				\foreach \i in {\first,...,\last} {
					\node (\i-bottom) at (2*\i*\componentwidth+\componentwidth, \height) {};
				}
				
				\foreach[evaluate={\j=int(\i-1)}] \i in {\second,...,\last} {
					\node (\i-top) at (2*\i*\componentwidth, 0) {};
					\draw (\i-top.south) -> (\i-bottom.north);
					\draw (\i-top.south) -> (\j-bottom.north);
				}
				
				\node (right) at (2*\last*\componentwidth+2*\componentwidth, 0) {\phantom{$[\,]$}};
				
				\path (\last-bottom) -- (right) node[midway] (mid-right) {};
				
				\node (left) at (2*\first*\componentwidth, 0) {\phantom{$[\,]$}};
				\path (\first-bottom) -- (left) node[midway] (mid-left) {};
				
				\draw[dashed] (\first-bottom.north) -> (mid-left);
				\draw[dashed] (\last-bottom.north) -> (mid-right);
			\end{scope}
			
			\begin{scope}
				\clip (2*\first*\componentwidth, \height-0.3) rectangle (2*\last*\componentwidth+2*\componentwidth, 0.3);
				
				\foreach \i in {\first,...,\last} {
					\node (\i-bottom) at (2*\i*\componentwidth+\componentwidth, \height) {};
				}
			
				\foreach[evaluate={\j=int(\i-1)}] \i in {\second,...,\last} {
					\node (\i-top) at (2*\i*\componentwidth, 0) {};
					\draw (\i-top.south) -> (\i-bottom.north);
					\draw (\i-top.south) -> (\j-bottom.north);
				}
				
				\node (right) at (2*\last*\componentwidth+2*\componentwidth, 0) {\phantom{$[\,]$}};
				
				\path (\last-bottom) -- (right) node[midway] (mid-right) {};
				
				\node (left) at (2*\first*\componentwidth, 0) {\phantom{$[\,]$}};
				\path (\first-bottom) -- (left) node[midway] (mid-left) {};
				
				\draw[dashed] (\first-bottom.north) -> (mid-left);
				\draw[dashed] (\last-bottom.north) -> (mid-right);
	
				\begin{scope}[circle]
					\node[fill=black, inner sep=1.8] at (0-bottom.north) {};
					\node[fill=black, inner sep=1.8] at (2-bottom.north) {};
					\node[fill=lightgray, draw, inner sep=1.8] at (1-bottom.north) {};
					\node[fill=lightgray, draw, inner sep=1.8] at (4-bottom.north) {};
					\node[fill=lightgray, draw, inner sep=1.8] at (-2-bottom.north) {};
				\end{scope}
			\end{scope}
		
			\begin{scope}[dashed, thick]
			    \begin{scope}[teal]
        			\draw (-3.5*\componentwidth, \height-0.4) rectangle (9.5*\componentwidth, 0.3);
        			
        			\node at (8*\componentwidth, \height-0.7) {$\Kshort$};
			    \end{scope}
    		
    		    \begin{scope}[violet]
        			\draw (-5.5*\componentwidth, \height-1.2) rectangle (11.5*\componentwidth, 0.7);
    
        			\node at (9.7*\componentwidth, \height-1.5) {$\Klong$};
    		    \end{scope}
    
                \begin{scope}[blue]
    			    \draw (-5.5*\componentwidth-0.3, \height-2) rectangle (11.5*\componentwidth+0.3, -\height+1.6);
        			\node at (11.4*\componentwidth, \height-2.3) {$K_W$};
                \end{scope}
    			
			\end{scope}

			\node[fill=white, below=-0.1 of 0-bottom] {$[a]$};
			\node[fill=white, below=-0.1 of 1-bottom] {$[bc]$};
			\node[fill=white, above=-0.2 of 1-top] {$[a|bc]$};
			
			\node[circle, fill=black, inner sep=1.8] (XW1) at (-5*\componentwidth, \height-2.6) {};
			\node[right=0 of XW1] {$\in X_W'$};

			\node[circle, fill=lightgray, draw, inner sep=1.8] (XW2) at (-2*\componentwidth, \height-2.6) {};
			\node[right=0 of XW2] {$\in \Ksparse \setminus X_W'$};
		\end{tikzpicture}
	\end{center}
	
	\caption{Example of how the subcomplexes $X_W' \subseteq \Ksparse \subseteq \Kshort \subseteq \Klong \subseteq K_W$ can intersect two different fiber components.
	The bottom fiber component intersects $X_W'$ (and thus also $\Ksparse$, $\Kshort$, and $\Klong$), whereas the top one does not.
	This figure visually motivates the names ``sparse'', ``short'', and ``long'' given to the three subcomplexes.
	The depicted situation can be realized e.g.\ if $W$ is the $(2, 3, \infty)$ triangle group (\Cref{fig:coxeter-axis}, right), the bottom component contains the simplices $\dotsc, [a], [a|bc], [bc], \dotsc$, and on top is any component that does not intersect $\Ksparse$ (see \Cref{rmk:components-intersecting-XW''}).
	}
	\label{fig:subcomplexes}
\end{figure}

The following is the chain of subcomplexes we have introduced:
\[
    X_W' \subseteq \Ksparse \subseteq \Kshort \subseteq \Klong \subseteq K_W.
\]
The definition of $\Kshort$ extends the definition of the \emph{canonical nice subcomplex} introduced in \cite[Section 7]{paolini2021proof} for the affine case, where the chain above simplifies to $X_W' = \Ksparse \subseteq \Kshort \subseteq \Klong = K_W$.
However, differently from the affine case, $\Ksparse$ (and thus $\Kshort$) does not necessarily intersect every fiber component of $K_W$ in general (see \Cref{sec:dmt-rank-three} below).
\Cref{fig:subcomplexes} exemplifies the definition of all subcomplexes; see also \cite[Figure 8]{paolini2021proof} for an affine example.

The reason for introducing the subcomplexes $\Kshort$ and $\Klong$ is twofold.
First, \Cref{prop:K''-K'} yields a deformation retraction $\Klong \searrow \Kshort$.
Second, the proof that $\Kshort \searrow \Ksparse$ in the affine case (\cite[Section 8]{paolini2021proof}) applies more generally, provided that one has an axial ordering of $R_0$ that induces an EL-labeling of $[1,w]$.

\section{Discrete Morse theory for the rank-three case}
\label{sec:dmt-rank-three}

We apply discrete Morse theory to the complexes introduced in the previous section in the case of a rank-three Coxeter system, which we assume to be irreducible and hyperbolic. The end goal of this section is to prove the following main theorem.

\begin{theorem}\label{thm:K-X'}
    The interval complex $K_W$ deformation retracts onto $X_W'$.
    In particular, $K_W$ is homotopy equivalent to the orbit configuration space $Y_W$.
    \label{thm:deformation-retraction}
\end{theorem}

The strategy of the proof is outlined in \Cref{treasuremap}. The only missing ingredients at this point are the deformation retractions $K_W\searrow \Klong$ and $\Kshort\searrow \Ksparse$, which we construct in \Cref{sec:matching,sec:acyclicity,sec:matching-N} below.

We will proceed in steps.  First, we will classify the fiber components of the complex $K_W$. Based on this classification, we will construct a discrete Morse matching $\M$ on $K_W$ that will prove the collapsing of $K_W$ onto $\Klong$. Then, we will explain how the arguments of \cite[Section 8]{paolini2021proof} carry over to our context, allowing us to prove the collapse of $\Kshort$ onto $\Ksparse$. Finally, we will summarize all the steps and prove \Cref{thm:K-X'}.

\subsection{Classification of fiber components}
We provide a geometric classification of the fiber components of $K_W$ for a rank-three hyperbolic Coxeter group.

\begin{lemma}
    Let $\CC$ be a $d$-fiber component of $K_W$, encoded by the bi-infinite sequence $(x_i)_{i \in \Z}$.
    Up to translation of the indices, exactly one of the following cases occurs.
    \begin{enumerate}[(i)]
        \item $\CC = \{ [\,], [w]\}$, $d=1$, and $x_i = w$ for all $i$.
        \item $d=2$, $x_{2i}$ is a vertical reflection and $x_{2i+1}$ is a rotation (possibly around a point at infinity).
        \item $d=2$, $x_{2i}$ is a horizontal reflection and $x_{2i+1}$ is a translation.
        \item $d=3$, all the $x_i$ are reflections, and at least one of the subsequences $(x_{3i})_{i \in \Z}$, $(x_{3i+1})_{i \in \Z}$, and $(x_{3i+2})_{i \in \Z}$ consists of vertical reflections.
        \item $d=3$ and all the $x_i$ are horizontal reflections.
    \end{enumerate}
    \label{lemma:fiber-components}
\end{lemma}

\begin{proof}
    The cases $d=1$ and $d=3$ are obvious. In the case $d=2$, we only need to use the fact that the (left or right) complement of a vertical reflection is a rotation, whereas the complement of a horizontal reflection is a translation (\Cref{lemma:reflections}).
\end{proof}

The classification of \Cref{lemma:fiber-components} has some similarities with the (rank-three) affine case (\cite[Section 7]{paolini2021proof}).
There, components of type (iii) are finite, type (v) does not occur, and overall there are only a finite number of components.
Here on the other hand, all components are infinite except for the single component of type (i). In addition, there are infinitely many components of types (iii) and (v).

\begin{remark}
Even though a reflection can never occur twice in the same minimal factorization of $w$, the sequence $(x_i)_{i \in \Z}$ defining a fiber component can have repetitions.
For example, if $t$ is the translation of \Cref{fig:five-lines} and $t=r_1r_2$ is its $\prec$-increasing factorization, then $r_2 = \varphi^3(r_1)$. Therefore, the fiber component of type (v) containing $[r_1|r_2]$ has a repetition of both $r_1$ and $r_2$.
\end{remark}

\begin{remark}
    The fiber components that intersect $\Ksparse$ (and thus constitute the subcomplex $\Klong$) are all those of types (i), (ii), (iv), and some of type (iii). Indeed, for every minimal element $[r_i\vert r_{i+1}]$ of a type (v) fiber component, the product $r_ir_{i+1}$ is a translation by \Cref{lem:reflections-vh} and fixes no vertices of the chamber $C_0$.
    \label{rmk:components-intersecting-XW''}
\end{remark}
If a fiber component of type (iii) intersects $\Ksparse$, we say that it is \emph{exceptional}.
Exceptional components are characterized in the following lemma.

\begin{lemma}
    Let $\CC$ be a fiber component of type (iii). The following are equivalent.
    \begin{enumerate}[(I)]
        \item $\CC$ intersects $\Ksparse$ (i.e., $\CC$ is exceptional).
        \item For every $[r] \in \CC$ where $r$ is a (horizontal) reflection, $r$ fixes an axial vertex.
        \item For every $[t] \in \CC$ where $t$ is a translation, there is at least one vertical reflection below $t$ in $[1,w]$.
    \end{enumerate}
    \label{lemma:exceptional-components}
\end{lemma}

\begin{proof}
    If (I) holds, then there exists a simplex $[r] \in \CC$ such that $r$ is a horizontal reflection that fixes a vertex of $C_0$.
    Then $\varphi^k(r)$ fixes an axial vertex for every $k \in \Z$, and thus (II) holds.
    Conversely, if (II) holds, then by \Cref{lemma:axial-vertices-orbits} there is at least one horizontal reflection $r$ that fixes a vertex of $C_0$ with $[r] \in \CC$, so (I) holds.
    
    If (II) holds, then by \Cref{lemma:rotations}, for any horizontal reflection $r$ with $[r] \in \CC$ we have that $r \leq u$ for some rotation $u \in [1, w]$ (around an axial vertex fixed by $r$).
    Therefore $t \geq r'$ where $t$ is the right complement of $r$ (a translation) and $r'$ is the right complement of $u$ (a vertical reflection). Thus (III) holds.
    Finally, if (III) holds and $t$ is a translation with $[t] \in \CC$, there is a vertical reflection $r' \leq t$. Then $u \geq r$, where $u$ is the right complement of $r'$ (a rotation around an axial vertex) and $r$ is the right complement of $t$ (a horizontal reflection).
    In particular, $r$ fixes an axial vertex. Therefore (II) holds.
\end{proof}

\subsection{Componentwise construction of the matching $\M$}
\label{sec:matching}

We now describe a perfect matching $\M$ (i.e., a matching with no critical simplices) on $K_W \setminus \Klong$.
Recall that $K_W \setminus \Klong$ consists of the union of all fiber components of type (v) and all non-exceptional components of type (iii).

Fix any point $p_0 \in \ell$ and consider the semi-open segment $I \subseteq \ell$ between $p_0$ (included) and $w(p_0)$ (excluded).
Then $I$ is a fundamental domain for the action of $\Z = \< w \>$ on $\ell$. Recall from \Cref{defxi} that the function $\xi$ assigns to every reflection $r$ the point $\xi(r)$ on the Coxeter axis that is closest to $\Fix(r)$.

\begin{definition}
    A translation $t \in [1,w]$ is \emph{special} if its left complement $r$ satisfies $\xi(r) \in I$.
\end{definition}

We could alternatively define a translation to be special if its \emph{right} complement $r$ satisfies $\xi(r) \in I$, or if the translation axis of $t$ intersects $\ell$ in a point of $I$.
All these definitions are equivalent up to changing the point $p_0 \in \ell$ in the definition of $I$.

\begin{remark}\label{oneandonly}
By construction, for every translation $t \in [1,w]$, there is exactly one $j\in \Z$ such that the translation $\varphi^j(t)$ is special.
Therefore, every fiber component of type (iii) contains exactly one simplex $[t]$ such that $t$ is a special translation.
For the same reason, every component of type (v) contains exactly three simplices of the form $[r_i | r_{i+1}]$ such that $r_i r_{i+1}$ is a special translation.
\end{remark}

\subsubsection{Matching on fiber components of type (v).}
\label{spec5}
Let $\CC$ be a fiber component of type (v), defined by a bi-infinite sequence of horizontal reflections $(r_i)_{i \in \Z}$.
Consider the sequence of axial order relations ($\prec$ or $\succ$) between consecutive reflections along the sequence $(r_i)_{i \in \Z}$.

\begin{lemma}
    Up to translation of the indices, the sequence of axial order relations among the reflections defining a component of type (v) is
    \[ \dotsb \prec r_{-1} \succ r_0 \succ r_1 \prec r_2 \succ r_3 \succ \dotsb, \]
    i.e., an infinite $3$-periodic repetition of one $\prec$ and two $\succ$.
    \label{lemma:order-reflections-type-v}
\end{lemma}

\begin{proof}
    By \Cref{lemma:phi-preserves-order}, the sequence of order relations is $3$-periodic: $r_i \prec r_{i+1}$ if and only if $r_{i+3} \prec r_{i+4}$.
    In addition, there cannot be two consecutive $\prec$, because then we would have an increasing factorization of $w$ into three horizontal reflections, contradicting \Cref{thm:shellability} (the only increasing factorization of $w$ is induced by the fundamental chamber $C_0$ as in \Cref{lemma:axial-factorization} and uses at least two vertical reflections).
    Finally, if we have $\dotsb \succ r_{-1} \succ r_0 \succ r_1 \succ r_2 \succ \dotsb$ (an all-$\succ$ sequence of order relations), then $r_0 \succ r_3 = \varphi(r_0)$ which is impossible by \Cref{lemma:phi-and-order}.
\end{proof}

Apply the previous lemma and assume that $r_{3i+1} \prec r_{3i+2}$ for all $i \in \Z$.
By \Cref{oneandonly}, among all translations $t_i = r_{3i+1} r_{3i+2}$, exactly one is special.
Without loss of generality, assume that $t_0 = r_1 r_2$ is special.
Then we match simplices of $\CC$ as follows:

\begin{leftbar}
\noindent On the component $\CC$, the matching $\M$ is defined as the unique acyclic matching that has $[r_1 | r_2]$ as the only critical simplex (see \Cref{fig:matching-M}, top).
\end{leftbar}

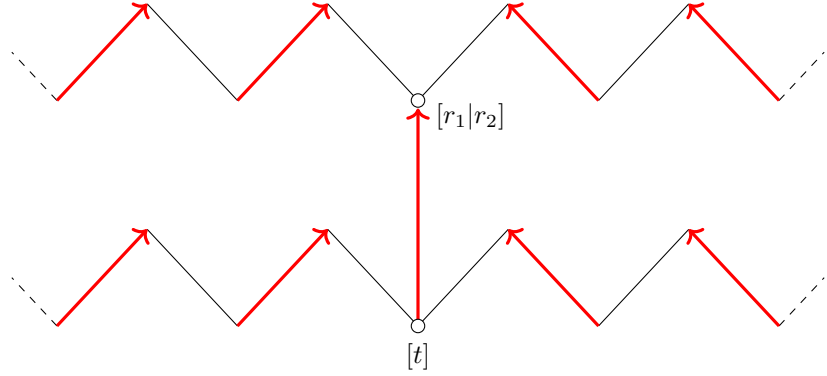
\begin{figure}
	\begin{center}
		\newcommand{\componentwidth}{1.2}
		\newcommand{\height}{-1.5}
		\begin{tikzpicture}
			\clip (-4*\componentwidth, -5) rectangle (6*\componentwidth, 0.3);  %

			\begin{scope}[inner sep=0.1cm]
				\node (0) at (0,0) {};
				\node (1) at (\componentwidth, \height) {};
				\node (2) at (2*\componentwidth, 0) {};
				\node (3) at (3*\componentwidth, \height) {};
				\node (4) at (4*\componentwidth, 0) {};
				\node (5) at (5*\componentwidth, \height) {};
				\node (6) at (6*\componentwidth, 0) {\phantom{$[\,]$}};
				\node (-1) at (-\componentwidth, \height) {};
				\node (-2) at (-2*\componentwidth, 0) {};
				\node (-3) at (-3*\componentwidth, \height) {};
				\node (-4) at (-4*\componentwidth, 0) {\phantom{$[\,]$}};
			\end{scope}

			\coordinate (5a) at ($(5)!0.5!(6)$) {};
			\coordinate (-3a) at ($(-3)!0.5!(-4)$) {};

			\node at (\componentwidth+0.7, -1.6) {$[r_1|r_2]$};

			\draw (0.south) -> (1.north);
			\draw (2.south) -> (1.north);
			\draw (4.south) -> (3.north);
			\draw[dashed] (5.north) -> (5a);
			
			\draw (0.south) -> (-1.north);
			\draw (-2.south) -> (-1.north);
			\draw[dashed] (-3.north) -> (-3a);
			
			\begin{scope}[red, very thick]
				\draw[<-] (-2.south) -- (-3.north);
				\draw[<-] (2.south) -- (3.north);
				\draw[<-] (0.south) -- (-1.north);
				\draw[<-] (4.south) -- (5.north);
			\end{scope}
			
			\begin{scope}[shift={(0, -3)}]a
				\begin{scope}[inner sep=0.1cm]
					\node (0bis) at (0,0) {};
					\node (1bis) at (\componentwidth, \height) {};
					\node (2bis) at (2*\componentwidth, 0) {};
					\node (3bis) at (3*\componentwidth, \height) {};
					\node (4bis) at (4*\componentwidth, 0) {};
					\node (5bis) at (5*\componentwidth, \height) {};
					\node (6bis) at (6*\componentwidth, 0) {\phantom{$[\,]$}};
					\node (-1bis) at (-\componentwidth, \height) {};
					\node (-2bis) at (-2*\componentwidth, 0) {};
					\node (-3bis) at (-3*\componentwidth, \height) {};
					\node (-4bis) at (-4*\componentwidth, 0) {\phantom{$[\,]$}};
				\end{scope}

				\coordinate (5abis) at ($(5bis)!0.5!(6bis)$) {};
				\coordinate (-3abis) at ($(-3bis)!0.5!(-4bis)$) {};
				
				\node at (\componentwidth, \height-0.3) {$[t]$};
				
				\draw (0bis.south) -> (1bis.north);
				\draw (2bis.south) -> (1bis.north);
				\draw (4bis.south) -> (3bis.north);
				\draw[dashed] (5bis.north) -> (5abis);
				
				\draw (0bis.south) -> (-1bis.north);
				\draw (-2bis.south) -> (-1bis.north);
				\draw[dashed] (-3bis.north) -> (-3abis);

				\begin{scope}[red, very thick]
					\draw[<-] (-2bis.south) -- (-3bis.north);
					\draw[<-] (2bis.south) -- (3bis.north);
					\draw[<-] (0bis.south) -- (-1bis.north);
					\draw[<-] (4bis.south) -- (5bis.north);
					
					\draw[<-] (\componentwidth, 1.5) -- (1bis.north);
				\end{scope}

				\begin{scope}[circle]
					\node[fill=white, draw, inner sep=1.8] at (1.north) {};
					\node[fill=white, draw, inner sep=1.8] at (1bis.north) {};
				\end{scope}
			\end{scope}
			
		\end{tikzpicture}
	\end{center}
	
	\caption{The matching $\M$ on a fiber component of type (v) (top) and the corresponding non-exceptional component of type (iii) (bottom), as described in \Cref{spec5,spec3,cfe}.
	The translation $t$ is special and $t=r_1r_2$ is its increasing factorization.}
	\label{fig:matching-M}
\end{figure}

\subsubsection{Matching on non-exceptional fiber components of type (iii).}
\label{spec3}
Let $\CC$ be a non-exceptional fiber component of type (iii). By \cref{oneandonly}, $\CC$ has a unique simplex $[t]$ such that $t$ is a special translation.

\begin{leftbar}
\noindent On the component $\CC$, the matching $\M$ is defined as the unique acyclic matching that has $[t]$ as the only critical simplex (see \Cref{fig:matching-M}, bottom).
\end{leftbar}

\subsubsection{Cross-fiber edges}
\label{cfe}
Let $\CC$ be a non-exceptional fiber component of type (iii) and let $[t]$ be the unique simplex of $\CC$ such that $t$ is a special translation. Then $t$ has a unique increasing factorization $t=r_1r_2$.
By \Cref{lemma:exceptional-components}, the reflections $r_1$ and $r_2$ are horizontal and so $[r_1|r_2]$ belongs to a fiber component of type (v). Notice that, so far, in the construction of $\M$ the simplices $[t]$ and $[r_1|r_2]$ have not been matched yet.

\begin{leftbar}
    \noindent Add to $\M$ all edges $[t]\rightarrow [r_1|r_2]$ where $t$ is a special translation in a non-exceptional component of type (iii) and $t = r_1 r_2$ is its increasing factorization (see \Cref{fig:matching-M}).
\end{leftbar}

\begin{remark}
	The matching $\M$ has no critical cells in $K_W\setminus \Klong$. In fact, let $\CC$ be a fiber component of type (v) and consider the unique simplex $[r_1|r_2] \in \CC$ such that $r_1 \prec r_2$ and $r_1r_2$ is special.
	Since the increasing factorization $t = r_1r_2$ has two horizontal reflections, there is no vertical reflection below $t$, so $[r_1r_2]$ is in a non-exceptional component of type (iii) by \Cref{lemma:exceptional-components} and is matched with $[r_1|r_2]$.
\end{remark}

\subsection{Acyclicity and properness of $\M$}
\label{sec:acyclicity}

\def\weight{\omega}

The matching $\M$ is regular because it does not involve the $0$-cell $[\,]$, which is the only simplex of $K_W$ that is a non-regular face of some other simplex (see \Cref{sec:garside-structures}).
In order to prove the acyclicity and properness of $\M$, we will weight every matched cell by a real number in such a way that the weight (weakly) decreases along directed paths (this will suffice to prove acyclicity) and the weights form a discrete subset of the positive real numbers (which will allow us to prove properness by induction on the weight).

\newcommand{\str}[1]{\underline{t}(#1)}
\newcommand{\sre}[1]{\underline{r}(#1)}

Given a cell $\sigma$ in a fiber component of type (v) or in a non-exceptional component of type (iii), let $\sre{\sigma}$ be the left complement of the special translation associated with the unique cross-fiber edge exiting the fiber component of $\sigma$ (see \Cref{spec5,spec3,cfe}).
Then, construct a poset map $\weight \colon \F(K_W) \setminus \F(\Klong) \to \R$ by setting
\[ \weight(\sigma)= d(\Fix(\sre{\sigma}),\ell), \]
where $d(\cdot,\cdot)$ denotes the distance between two lines in the hyperbolic plane.
Note that $\weight$ is constant on fiber components and $\sre{\sigma}$ is always a horizontal reflection.

\begin{lemma}
	The image of $\omega$ is a discrete subset of the positive real numbers bounded away from $0$.
	\label{lemma:discrete-image}
\end{lemma}

\begin{proof}
	By \Cref{rmk:reflections-distance}, the values taken by $\weight$ are positive and bounded away from zero.
	We now show that, for every $\delta>0$, the intersection $\im(\omega) \cap (0, \delta)$ is finite.
	Let $\sigma$ be a cell such that $\weight(\sigma) < \delta$.
	By definition of special translation, we have that $\xi(\sre{\sigma}) \in I$ (recall that $I$ is introduced at the beginning of \Cref{sec:matching}).
	Therefore, $\Fix(\sre{\sigma})$ intersects the closed $\delta$-neighborhood $N$ of the segment $I$.
	Since $N$ is compact, it intersects only a finite number of reflection hyperplanes.
	In particular, $\weight(\sigma) = d(\Fix(\sre{\sigma}),\ell)$ can only take a finite number of possible values.
\end{proof}

\begin{lemma}
    The function $\weight$ is order-preserving. In addition, $\omega(\sigma) = \omega(\tau)$ whenever $[\tau, \sigma] \in \M$.
\end{lemma}

\begin{proof}
    Suppose that $\tau$ is a face of $\sigma$, with both simplices belonging to $\F(K_W) \setminus \F(\Klong)$.
    For the first part of the claim, we want to show that $\weight(\tau) \leq \weight(\sigma)$.
    The only non-trivial case is if $\tau$ and $\sigma$ belong to different components, and this only happens if $\tau$ is in a (non-exceptional) component of type (iii) and $\sigma$ is in a component of type (v).

    Let $(r_i)_{i \in \Z}$ be the sequence of (horizontal) reflections defining the fiber component $\CC$ of $\sigma$.
    By \Cref{lemma:order-reflections-type-v}, without loss of generality,
    \[ \dotsb \prec r_{-1} \succ r_0=\sre{\sigma} \succ r_1 \prec r_2 \succ r_3 \succ \dotsb. \]
    By Case 2 in the proof of \Cref{thm:shellability} applied to the (special) translation $t=r_1r_2$, the fact that $r_1\prec r_2$ implies that $r_1$ and $r_2$ must be the $\prec$-first and $\prec$-last reflections below $t$, and so the fixed lines of $r_1$ and $r_2$ are on opposite sides of $\ell$.  Since $r_0=\sre{\sigma}$ is the left complement of  the (special) translation $t=r_1r_2$, it follows immediately from \Cref{lemma:five-lines} that the distances $d(\Fix(r_1), \ell)$ and $d(\Fix(r_2), \ell)$ are both strictly lower than the distance $d(\Fix(r_0), \ell) = d(\Fix(r_3), \ell)$.
    Now recall that the distance $d(\Fix(\cdot), \ell)$ is invariant under conjugation by $w$, whence $\weight(\sigma)=d(\Fix(r_{3k}),\ell)$ and $d(\Fix(r_{i}),\ell)=d(\Fix(r_{3k+i}),\ell)$ for all $k$ and for $i=1,2$.
    This implies that $ \weight(\sigma) = \max \{ d(\Fix(r_i), \ell) \mid i \in \Z \} $.
    In addition, $\weight(\tau) = d(\Fix(r_i), \ell)$ for some $i \in \Z$ and therefore $\weight(\tau) \leq \weight(\sigma)$.

    For the second part of the claim, since $\weight$ is constant on fiber components, it is enough to check that $\weight$ is constant along cross-fiber edges of $\M$.
    Such edges are of the form $[t]\to [r_1\vert r_2]$ where $t=r_1r_2$ is a special translation. By definition of $\weight$, we have $\weight([t])=d(\Fix(wt^{-1}),\ell)=\weight([r_1\vert r_2])$.
\end{proof}

\begin{lemma}
    The matching $\M$ is acyclic.
    \label{lemma:acyclic-matching}
\end{lemma}

\begin{proof}
    By the Patchwork theorem (\Cref{thm:patchwork}),
    it is enough to show that there is no alternating cycle on which $\weight$ is constant.
    Suppose for the sake of contradiction that such an alternating cycle $\gamma$ exists.
    
    \emph{Case 1:} suppose that $\gamma$ has no cross-fiber edges.
    Then $\gamma$ is also an alternating cycle with respect to the smaller matching $\M' \subseteq \M$ where all cross-fiber edges are removed.
    Note that every edge of $\M'$ matches elements in the same fiber component.
    By the Patchwork theorem applied to the map $\eta$ of \Cref{defeta}, $\gamma$ needs to be entirely contained in a single fiber component.
    However, the restriction of $\M'$ to every fiber component is obviously acyclic.
    
    \emph{Case 2:} suppose that $\gamma$ has at least one cross-fiber edge $[t] \to [r_1|r_2]$, where $t$ is a special translation and $r_1r_2$ is its increasing factorization.
    After that edge, $\gamma$ has to continue either with $[r_1|r_2] \to [r_1]$ or $[r_1|r_2] \to [r_2]$. In both cases, the value of $\weight$ strictly decreases by \Cref{lemma:five-lines}, which is a contradiction.
\end{proof}

\begin{lemma}
    The matching $\M$ is proper.
    \label{lemma:proper-matching}
\end{lemma}

\begin{proof}
    Without using cross-fiber edges,
    from any simplex it is possible to reach only a finite number of simplices. Indeed, this is true while staying in the same fiber component, and any path starting from a $d$-simplex can only change the fiber component $d-1$ times.
    
    Therefore, it is enough to check that there are only a finite number of simplices reachable from $[t]$ where $t$ is any special translation.
    We prove this by induction on $\weight([t])$, which takes values in a discrete subset of the positive real numbers bounded away from $0$ (by \Cref{lemma:discrete-image}).
    A directed path starting from $[t]$ necessarily begins with $[t] \to [r_1|r_2] \to [r_i]$ where $i\in \{1,2\}$.
    Then it continues as an alternating path inside the fiber component $\CC_i$ of $[r_i]$ until it reaches the unique simplex $[t_i] \in \CC_i$ such that $t_i$ is a special translation.
    As already noted in the proof of \Cref{lemma:acyclic-matching}, we have that $\weight([t_i]) = \weight([r_i]) < \weight([t])$ by \Cref{lemma:five-lines}.
    By induction, there are only a finite number of simplices reachable from $[t_1]$ and $[t_2]$, so there are only a finite number of simplices reachable from $[t]$.
\end{proof}

\subsection{Construction of the matching $\mathcal N$}
\label{sec:matching-N}

\def\MM{\mathcal{N}}
\def\mf{\mu}

In this section, we construct an acyclic and proper matching $\MM$ on $\Kshort$ whose critical cells are exactly the faces of $ \Ksparse$.
The construction of $\MM$, as well as the proof of acyclicity and properness, closely follow the treatment given in \cite[Section 8]{paolini2021proof} for the case of affine Artin groups,
but notation differs slightly: our complexes $\Ksparse$ and $\Kshort$ respectively take the place of $X'_W$ and $K_W'$ in \cite{paolini2021proof} where, accordingly, a matching on $\F(K_W')\setminus \F(X'_W)$ is constructed (note that $X_W' = \Ksparse$ in the affine case).

The matching $\MM$ will be defined as the set of orbits of an involution $\mf$ on $\F(\Kshort)\setminus \F(\Ksparse)$ (see \Cref{def:NN} below).
For the remainder of this subsection, we consider  simplexes of the form
\[
	\sigma=[x_1\vert x_2 \vert \cdots \vert x_k] \in \F(\Kshort)\setminus \F(\Ksparse)
\]
and we set $\pi(\sigma):=x_1\cdots x_k \in [1,w]$. %

\begin{definition}
   	Let $\sigma=[x_1\vert \cdots \vert x_d]$ be such that  $\pi(\sigma)=w$.
	The depth $\delta = \delta(\sigma)$ is the minimum $i \in \{ 1, 2, \dotsc, d\}$ such that either $x_i$ has reflection length at least $2$ or else $i\leq d-1$ and $x_i$ precedes in the axial ordering every reflection  that is below $x_{i+1}$  in $[1,w]$. If no such $i$ exists, set $\delta(\sigma)=\infty$.
\end{definition}

\begin{lemma}
    Let $\sigma=[x_1\vert \cdots \vert x_d]$ be such that  $\pi(\sigma)=w$ and  $x_2\cdots x_d$ fixes some vertex of $C_0$. Then $\delta(\sigma)<\infty$.
    \label{lemma:delta-not-infinity}
\end{lemma}

\begin{proof}
    This is the claim of \cite[Lemma 8.4]{paolini2021proof}, whose proof  
    relies on \cite[Lemma 8.2]{paolini2021proof} and \cite[Proposition 3.10 and Remark 3.2]{paolini2021proof}. These correspond to our Lemmas \ref{lemma:smallest-reflections} and \ref{lemma:axial-vertices-orbits}, respectively.
\end{proof}

\begin{definition}[The involution $\mu$]
    Recall the map $\eta$ from \eqref{eq:eta} and let $d = \eta(\sigma)$, so that $\sigma = [x_1\vert \dotsb \vert x_k]$ belongs to a $d$-fiber component $\CC$.
    \begin{enumerate}
        \item If $\pi(\sigma)\neq w$, then $k=d-1$ and we let $\mf(\sigma)$ be the simplex $[x \vert x_1\vert \cdots \vert x_{d-1}] \in \CC$ that lies immediately to the left of $\sigma$.
    
        \item If $\pi(\sigma)= w$ (so $k=d$) and $x_2\cdots x_d$ does not fix any vertex of $C_0$, let $\mf(\sigma) = [x_2\vert \cdots \vert x_d]$.
    \end{enumerate}
    Suppose now that $\pi(\sigma) = w$ (so $k=d$) and $x_2\cdots x_d$ fixes a vertex of $C_0$. Let $\delta = \delta(\sigma)$ and note that $\delta \neq \infty$ by \Cref{lemma:delta-not-infinity}.
    \begin{enumerate}
        \setcounter{enumi}{2}
        \item If $x_\delta$ is not a reflection, let $\mf(\sigma) = [x_1|\cdots |x_{\delta -1}|y|z|x_{\delta+1}| \cdots |x_d]$, where $x_\delta= yz$ and $y$ is the $\prec$-smallest reflection below $x_\delta$ in $[1,w]$.
    
        \item If $x_\delta$ is a reflection, let $\mf(\sigma) = [x_1|\cdots |x_{\delta}x_{\delta+1}| \cdots |x_d]$.
    \end{enumerate}
\end{definition}

\begin{lemma}
    The map $\mu$ is well-defined and involutive on $\F(\Kshort)\setminus \F(\Ksparse)$.
\end{lemma}

\begin{proof}
    The claim can be proved by the same arguments as \cite[Lemma 8.8 and Proposition 8.9]{paolini2021proof}, using properties of vertical reflections proved in \Cref{lemma:conjugate-smallest-reflections} and \Cref{lemma:smallest-reflections} (which replace \cite[Lemmas 8.1 and 8.2]{paolini2021proof}), the order-preserving correspondence between elements of $[1,w]$ and linear subspaces of $\R^3$ proved in \Cref{lemma:mov} (which replaces \cite[Lemmas 2.15 and 2.16]{paolini2021proof}), and \Cref{lemma:rotations} (which replaces \cite[Lemma 3.18]{paolini2021proof}).
    We also use the fact that the natural edge-labeling of $[1,w]$ is an EL-labeling with respect to the axial ordering (\Cref{thm:shellability}).
\end{proof}

\begin{definition}
    The matching $\MM$ on $\F(\Kshort)\setminus \F(\Ksparse)$ consists of all pairs of cells $(\mu(\sigma), \sigma)$ such that $\mu(\sigma)$ is a face of $\sigma$.
    \label{def:NN}
\end{definition}

\begin{lemma}
    The matching $\MM$ is proper and acyclic.
\end{lemma}

\begin{proof}
    The main tool consists in a special total ordering $\triangleleft$ on the set of all minimal length factorizations of $w$ as a product of reflections. This total ordering is induced by the axial ordering $\prec$ as defined in \cite[pp.\ 548-549]{paolini2021proof}.
    Now, given any simplex $\sigma=[x_1|x_2|\cdots |x_k] \in \F(\Kshort)\setminus \F(\Ksparse)$, we can consider the minimal length factorization of $w$ obtained by concatenating the $\prec$-increasing factorizations of $y, x_1, x_2, \dotsc, x_k$, where $y$ is the left complement of $x_1x_2\dotsm x_k$ (such increasing factorizations are unique because of the EL property).
    This defines an order-preserving map from $\F(\Kshort)\setminus \F(\Ksparse)$ to the (totally) $\triangleleft$-ordered set of minimal length factorizations of $w$.
    Properness and acyclicity of $\MM$ follow from the fact that this order-preserving map is strictly increasing along alternating paths of $\MM$. This can be proved as in \cite[Lemmas 8.10-8.13]{paolini2021proof} with (extensive) use of the EL property of the axial ordering (\Cref{thm:shellability}).
\end{proof}

\subsection{Proof of \Cref{thm:K-X'}}

\begin{proof}[Proof of \Cref{thm:K-X'}]
    By \Cref{lemma:acyclic-matching,lemma:proper-matching}, $\M$ is a proper acyclic matching on $K_W$ having $\Klong$ as the subcomplex of critical simplices.
    Using discrete Morse theory (\Cref{thm:dmt}), we deduce that $K_W \searrow \Klong$.
    \Cref{prop:K''-K'} shows that $\Klong \searrow \Kshort$.
    Discrete Morse theory applied to the matching $\MM$ of \Cref{sec:matching-N} shows that $\Kshort \searrow \Ksparse$. Combining these collapses we obtain a collapse $K_W \searrow \Ksparse$.
    In addition, $K_{W_T} \searrow \Ksparse[W_T]$ for all infinite proper standard parabolic subgroups $W_T$ by \cite[Theorem 8.14]{paolini2021proof} (such subgroups are necessarily affine of type $\tilde A_1$).
    Finally, \Cref{prop:inductive-collapse} implies that $K_W \searrow X_W'$.
\end{proof}

\section{Consequences for rank-three Artin groups}
\label{sec:consequences}

In this section, we use the theory developed in the rest of the paper to derive several results about rank-three Artin groups.
All results are already known in the spherical and affine cases, so in our proofs below we implicitly restrict ourselves to the hyperbolic case.

Using the deformation retraction $K_W \searrow X_W'$, we are now able to establish the isomorphism between the standard and the dual Artin groups associated with $W$, as well as the $K(\pi, 1)$ conjecture in the rank-three case. Note that the $K(\pi, 1)$ conjecture was already proved by other means by Charney and Davis \cite{charney1995k} for $2$-dimensional Artin groups, which include rank-three Artin groups as a special case.

\begin{theorem}
    If $W$ has rank three, then the natural map from the standard Artin group $G_W$ to the dual Artin group $W_w$ is an isomorphism.
    \label{thm:standard-dual-isomorphism}
\end{theorem}

\begin{proof}
    The natural map $G_W \to W_w$ is induced by the inclusion of the subcomplex $X_W'$ into $K_W$, which is a homotopy equivalence by \Cref{thm:deformation-retraction}.
\end{proof}

\begin{theorem}
    The $K(\pi, 1)$ conjecture holds for all rank-three Artin groups $G_W$.
    \label{thm:conjecture}
\end{theorem}

\begin{proof}
    By \Cref{thm:deformation-retraction}, the orbit configuration space $Y_W$ is homotopy equivalent to $K_W$, which is a classifying space by \Cref{thm:lattice,thm:garside}.
\end{proof}

In light of the isomorphism $G_W \cong W_w$, we can now use the Garside structure of $W_w$ to study the standard Artin group $G_W$.

\begin{theorem}
    Rank-three Artin groups are Garside groups.
    \label{thm:standard-garside}
\end{theorem}

\begin{proof}
    This is an immediate consequence of \Cref{thm:standard-dual-isomorphism} and \Cref{cor:dual-garside}.
\end{proof}

In particular, using the dual Garside structure, we can solve the word problem and easily prove that the center is trivial (except in the spherical cases).

\begin{theorem}
    Rank-three Artin groups have a solvable word problem.
    \label{thm:word-problem}
\end{theorem}

\begin{proof}
    By \Cref{thm:garside}, it is enough to be able to check equality and compute meets and joins in $[1,w]$.
    Equality can be checked using a solution for the word problem in $W$, expressing elements as a product of simple reflections (see for example \cite[Chapter 4]{bjorner2006combinatorics}).
    Since $[1,w]$ is self-dual, it is enough to be able to compute joins.
    
    The only non-trivial case to consider is the join of two distinct reflections $r_1$ and $r_2$.
    Representing $W$ as a group of isometries of the hyperbolic plane, compute the intersection point of $\Fix(r_1)$ and $\Fix(r_2)$. If such a point $p$ exists (possibly at infinity), then the join $r_1 \vee r_2$ is either the only rotation $u \in [1,w]$ fixing $p$ (if $p$ is an axial vertex) or $w$ otherwise.
    To check if $p$ is an axial vertex, check whether $p$ coincides with one of the vertices of the axial chambers whose closure contains the projection of $p$ onto the Coxeter axis $\ell$.
    
    Suppose now that $\Fix(r_1)$ and $\Fix(r_2)$ do not intersect in $\H^2 \cup \partial \H^2$.
    The join $r_1 \vee r_2$ is either the only translation $t \in [1, w]$ whose axis is the line $l$ orthogonal to both $\Fix(r_1)$ and $\Fix(r_2)$ (if such a translation exists in $[1,w]$) or $w$ otherwise.
    There are only a finite number of reflections $r \in W$ whose fixed line intersects $l$ between $\Fix(r_1)$ and $\Fix(r_2)$ (included). By \Cref{lemma:reflections-below-translation}, the translation $t$ (if it exists) has to be equal to $rr_1$ or $r_1r$ for at least one of these reflections $r$.
    Then we can consider all elements of this form and check whether one of them is a translation $t$ such that $wt^{-1}$ is a reflection (this is equivalent to $t$ being in $[1,w]$, because $t$ has reflection length $2$).
    If we find such a translation $t \in [1, w]$, then $t = r_1 \vee r_2$ because the axis of $t$ is orthogonal to both $\Fix(r_1)$ and $\Fix(r_2)$ by the first part of \Cref{lemma:reflections-below-translation}.
\end{proof}

\begin{theorem}
    Non-spherical rank-three Artin groups have a trivial center.
    \label{thm:center}
\end{theorem}

\begin{proof}
    Let $x$ be an element of the center of a rank-three Artin group $G_W$.
    The dual Garside structure yields a normal form $x = \iota(w)^{-m} \iota(u_1) \iota(u_2) \dotsm \iota(u_k)$, where $u_1, \dotsc, u_k \in [1,w]$ and $\iota\colon [1,w] \to G_W$ is the natural immersion.
    A well-known property of the normal form is that $x$ commutes with $\iota(w)$ if and only if all $u_i$'s commute with $w$ in $W$ \cite[Proposition 2.14]{mccammond2017artin}.
    In particular, this has to hold since $x$ is in the center of $G_W$.
    However, for any $u \in [1, w]$ different from $1$ and $w$, no non-trivial power of $w$ commutes with $u$: if $u$ is a reflection, then $\xi(w^n u w^{-n}) = w^n \xi(u) \neq \xi(u)$; passing to the (left) complement, we reach the same conclusion if $u$ is a rotation or a translation.
    Therefore, $x$ is a power of $\iota(w)$.
    Thus $x$ does not commute with any $\iota(u)$ where $u\in [1,w] \setminus \{1, w \}$, unless $x = 1$.
\end{proof}

Note that the word problem was already solved by Chermak \cite{chermak1998locally} for locally non-spherical Artin groups, a class that includes rank-three Artin groups.
In a preprint by Jankiewicz and Schreve \cite{jankiewicz2022conjecture}, posted while this article was in preparation, it is shown that the $K(\pi, 1)$ conjecture implies that the center is trivial (for a non-spherical irreducible Artin group), thus providing another proof that non-spherical rank-three Artin groups have a trivial center.

\bibliographystyle{amsalpha-abbr}
\bibliography{bibliography}

\end{document}